%% file: PAT_revision.tex
\documentclass{aims} 
\usepackage{amsmath,amssymb}
\usepackage{paralist}
\usepackage[misc]{ifsym}
\usepackage{epsfig} 
\usepackage{epstopdf} 
\usepackage[colorlinks=true]{hyperref}
\hypersetup{urlcolor=blue, citecolor=red}
\allowdisplaybreaks

\def\currentvolume{X}
 \def\currentissue{X}
  \def\currentyear{200X}
   \def\currentmonth{XX}
    \def\ppages{X-XX}
     \def\DOI{10.3934/xx.xxxxxxx}

\newtheorem{theorem}{Theorem}[section]
\newtheorem{corollary}[theorem]{Corollary}

\newtheorem{lemma}[theorem]{Lemma}
\newtheorem{proposition}[theorem]{Proposition}

\theoremstyle{definition}
\newtheorem{definition}[theorem]{Definition}
\newtheorem{remark}[theorem]{Remark}
\newtheorem{assumption}[theorem]{Assumption}

\newcommand{\R}{\mathbb{R}}

\newcommand{\W}{\mathcal{W}}
\renewcommand{\L}{\mathbb{L}}

\newcommand{\va}[1]{\left\{\begin{array}{r@{\text{ }}ll}#1\end{array}\right.}
\newcommand{\inner}[1]{\left ( #1 \right)}
\newcommand{\dual}[1]{\left\langle #1 \right\rangle}
\newcommand{\norm}[1]{\left\| #1\right\|}
\renewcommand{\d}{\mathrm{d}}

\newcommand{\iii}{i}
\newcommand{\coeffi}{d}
\newcommand{\Hone}{X^\Gamma}

\newcommand{\bld}[1]{\boldsymbol{#1}}
\newcommand{\bo}[1]{\mathbf{#1}}
\newcommand{\un}[1]{\mathrm{\,#1}}

\usepackage{tcolorbox}

\newcommand{\revision}[2]{#2}
\graphicspath{{./figs/}}

\DeclareMathOperator{\tr}{tr}
\DeclareMathOperator{\Tr}{Tr}

\DeclareMathOperator{\HS}{HS}

\DeclareMathOperator{\diag}{diag}

\DeclareMathOperator{\spann}{span}

\DeclareMathOperator{\post}{post}
\DeclareMathOperator{\pr}{pr}
\DeclareMathOperator{\var}{var}

\title[Optimum experiment design for photoacoustic tomography]
{On the optimal choice of the illumination function in photoacoustic tomography} 

\author[Phuoc-Truong Huynh and Barbara Kaltenbacher]{}


\thanks{This research was funded in part by the Austrian Science Fund (FWF) [10.55776/DOC78]}

\thanks{$^*$Corresponding author: Barbara Kaltenbacher}

\begin{document}
\maketitle

\centerline{\scshape
Phuoc-Truong Huynh$^{{\href{mailto:phuoc.huynh@aau.at}{\textrm{\Letter}}}1}$
and Barbara Kaltenbacher$^{{\href{mailto:barbara.kaltenbacher@aau.at}{\textrm{\Letter}}}*1}$}

\medskip

{\footnotesize
 \centerline{$^1$Department of Mathematics, University of Klagenfurt, Austria}
} 



\bigskip

 \centerline{(Communicated by Handling Editor)}


\begin{abstract}
This work studies the inverse problem of photoacoustic tomography (more precisely, the acoustic subproblem) as the identification of a space-dependent source parameter.  The model consists of a wave equation involving a time-fractional damping term to account for power law frequency dependence of the attenuation, as relevant in ultrasonics. We solve the inverse problem in a Bayesian framework using a Maximum A Posteriori (MAP) estimate, and for this purpose derive an explicit expression for the adjoint operator.  
On top of this, we consider optimization of the choice of the laser excitation function,  which is the time-dependent part of the source in this model, to enhance the reconstruction result.  The method employs the $A$-optimality criterion for Bayesian optimal experimental design with Gaussian prior and Gaussian noise.  To efficiently approximate the cost functional, we introduce an approximation scheme based on projection onto finite-dimensional subspaces.  Finally, we present numerical results that illustrate the theory.
\end{abstract}


\section{Introduction}

Photoacoustic tomography (PAT) is a novel non-invasive imaging technique that combines optical and ultrasound methods to generate high-resolution images of biological tissue, and has recently attracted much  interest.  In PAT, short pulses of laser light are emitted into the tissue.  This causes a rapid thermal expansion and generates ultrasonic waves through the photoacoustic effect.  The laser-induced ultrasonic waves propagate through the tissue and are captured by an array of detectors surrounding the tissue's surface. These measurements are then processed to reconstruct the optical absorption properties of the tissue, providing detailed information about its internal structure \cite{Treeby_Zhang_Cox_2010,  Kuchment_Kunyansky_2011,  Xia_Yao_Wang_2014}.  

The propagation of acoustic pressure in PAT is typically described by \textit{acoustic wave equations} of the form  
\begin{equation}\label{eq:wave_general}
c_0^{-2}\partial^2_t p - \Delta p = \partial_t J,
\end{equation}
on a space-time domain $(0,T)\times\Omega$, where $\Omega$ is a bounded domain in $\R^d$ or $\Omega = \R^d$, $d \in \{2, 3\}$.
The medium is assumed to be at rest initially and suitable boundary conditions are imposed in case $\Omega\not=\R^d$.  Here, $p = p(t,x)$ is the acoustic pressure at location $x$ and time $t$,   $c_0 = c_0(x)$ is the space-dependent sound speed, assumed to be positive, and $J = J(t,x)$ is the heating function defined as the thermal energy converted per unit volume and per unit time \cite{Wang_Wu_2007}.  Up to a positive factor,  the heating function $J$ is of the form $J(t,x) = i(t)a(x)$,  where $i$ represents the temporal profile of the \textit{illumination} used in the experiment and $a$ is the \textit{absorption density function} to be identified.  The inverse problems of photoacoustic tomography is to determine the absorption density function $a$,  using measurements of the acoustic pressure on an array of transducers represented by a surface $\Sigma$ immersed in the acoustic domain $\Omega$
\begin{equation}\label{eq:obs}
p=p_{\text{obs}}\text{ on }(0,T)\times\Sigma.
\end{equation}

Usually,  a very short laser pulse is used for illumination,  i.e.  $i(t) \approx \delta_0 (t)$ where $\delta_0$ is the Dirac function concentrated at $t=0$.  With this approximation, 
the model for PAT becomes
 \begin{equation}\label{eq:wave_undamped}
\va{ c_0^{-2}\partial^2_t p - \Delta p &= 0 &\text { in } (0, T) \times \Omega, \\
p(0,\cdot) = a(x), \partial_t p(0,\cdot) &= 0 &\text{ on }\Omega.
}    
\end{equation}
with boundary conditions.
From an inversion perspective, the unknown source term is treated as an initial condition, for which various reconstruction techniques have been explored in, e.g., \cite{Arridge_2016, Elbau_Scherzer_Shi_2017, Haltmeier_Nguyen_2017, Treeby_Zhang_Cox_2010, Scherzer_Kowar_2010}.
Nevertheless,  over the past decade,  intensity-modulated continuous-wave (IM-CW) lasers have gained popularity as a possibly preferred choice for optical excitation in PAT. While short pulses are acknowledged for yielding better results with the same light energy, IM-CW lasers demonstrate their advantages through their compactness and cost-efficiency \cite{Lang_2019,  Petschke_2010}.  
Considering a source identification formulation in time domain allows to capture a wide range of illumination functions, including delta pulse like and CW-like choices.

In addition, from a practical standpoint, attenuation effects are crucial, as neglecting them can lead to distortions and artifacts in the reconstructed images \cite{Scherzer_Kowar_2010, Haltmeier_Nguyen_2019}.  To address this issue,  we adopt an attenuated model for photoacoustic tomography in the time domain, employing a fractionally damped wave equation
\begin{equation}\label{eq:wave_damped}
\va{ c_0^{-2}\partial^2_t p - \Delta p - {b_0\partial_t^\alpha\Delta p} &= a(x) i'(t)&\text{ in } (0,T) \times \Omega,\\
p &= 0 &\text{ on }(0,T)\times \partial \Omega,\\
p(0,\cdot) = \partial_t p(0,\cdot) &= 0 &\text{ in }\Omega.}
\end{equation}
Here $\partial_t^\alpha$ is a time-fractional derivative of order $\alpha\in[0,1]$, cf. Section~\ref{sec:preliminaries}, $b_0$ derives from the diffusivity of sound and 
we assume
$$
{
c_0>0,\ b_0 \geq0 \text{ a.e. on }\Omega, \quad
c_0, \ c_0^{-1}, \ b_0, \ b_0^{-1} \ \in L^\infty(\Omega),  \quad 
c_0 \equiv c, \ b_0 \equiv b\text{ a.e. on }\Omega\setminus \Omega_0,
}
$$
that is, $b_0$ and $c_0$ take constant values outside the imaging domain $\Omega_0\subseteq\Omega$ and are bounded away from zero and infinity on $\Omega$.

Our aim in this work is to study the identification problem of the absorption density $a = a(x)$ in the fractionally damped model \eqref{eq:wave_damped}.  In addition,  we investigate the influence of the intensity function $i = i(t)$ on the reconstruction result.  More precisely in the context of optimal experimental design, we seek to determine an optimal laser intensity setup so that the quality of the reconstructed parameter is maximized in a certain sense.

\subsection*{Related works}
The problem of PAT has been attracting much interest during the last decade, both from a mathematical perspective and in its practical applications.  Various approaches have been employed to solve PAT,  for instance time reversal methods \cite{Hristova_Kuchment_Nguyen_2008, Treeby_Zhang_Cox_2010}, and direct methods involving the adjoint problem \cite{belhachmi_2016,  haltmeier_2017}.  Studies on the latter class of methods employ both discretize-then-adjoin \cite{Tick_2016} and adjoin-then-discretize approaches \cite{Arridge_2016,  belhachmi_2016}, where the adjoint operator is explicitly described in its continuous form and remains independent of the discretization process. 

Attenuation models have been extensively studied to account for the loss of acoustic energy as photoacoustic waves travel through tissues or other media. These models are often formulated in the frequency domain, where attenuation causes high-frequency components to dissipate more quickly over distance,  see \cite{Scherzer_Kowar_2010,  Elbau_Scherzer_Shi_2017}.  To derive the corresponding lossy wave equation in the time domain,  fractional derivatives in time, which also capture memory effects,  are involved \cite{Holm2019,Szabo_1994}.  Inverse problems with fractional wave models have been investigated \cite{Kaltenbacher_Rundell_2021},  while efficient numerical methods for solving fractionally damped wave equations are explored in studies such as \cite{kaltenbacher_2022, Baker_Banjai_2022}; we also refer to the recent book \cite{JinZhi2023}.  These operators effectively capture the frequency-dependent attenuation by modeling the dispersive and dissipative behavior of wave propagation over time,  see also \cite{Kaltenbacher_Rundell_2023} for details.

Tuning the choice of the intensity function in the model falls within the scope of optimal experimental design (OED) for parameter identification.  There exists an extensive amount of literature focused on OED for various types of parameter identification problems.  OED for infinite-dimensional Bayesian inverse problems has been introduced in \cite{Alexanderian_2014, Alexanderian_2016} with a focus on the two most common criteria,  namely $A$- and $D$-optimality criteria.  In the particular setting with Gaussian noise and Gaussian prior distribution for the parameter, the $A$-optimality criterion minimizes the trace of the posterior covariance matrix,  while the $D$-optimality criterion \revision{determinant of the posterior precision matrix (inverse covariance)}{expected information gain}.  Detailed interpretation of these criteria can be found in \cite{Alexanderian_2021}.  It is known that OED for infinite-dimensional inverse problems is challenging from both mathematical and computational points of view.  
Computing the cost functional in \eqref{eq:A_optimal} requires the trace of the inverse of an an operator whose dimension is typically very large \cite{Alexanderian_2016, koval_2020}. 
Hence,  efficient methods to solve these problems have been addressed by approaches typically relying on low-rank approximation of the forward operator.  Another option is to use a random estimator for the trace in \eqref{eq:A_optimal}, see \cite{Alexanderian_2016}, or trace on the observation space \cite{koval_2020}. We refer to \cite{Alexanderian_2021} for a detailed review.  From a practical perspective, application of OED to various problems has been explored,  
for instance 
inference of the permeability in porous medium flow \cite{Alexanderian_2016},
source identification in contaminant transport \cite{Alexanderian_etal2021},
iron loss in electrical machines \cite{Hannukainen_Hyvonen_Perkkio_2021},  
magnetorelaxometry imaging \cite{Helin_Hyvonen_Maaninen_Puska_2023}, 
stellarator coil design \cite{Giuliani_etal2022},
and electrical impedance tomography \cite{Hyvonen_2023}. \revision{}{Notably, the OED problem for (quantitative) photoacoustic tomography has also been studied in a recent work \cite{scope_anastasio_villa_2024} based on the Bayesian Cram\'{e}r-Rao bound, where the focus is on optimizing the illumination system configuration, including factors such as the positioning and rotation of cone-beam light sources.}

\subsection*{Contributions} The main contributions of this work are twofold.
\begin{enumerate}
\item[(1)] We study the problem of photoacoustic tomography in lossy media in the time domain. The parameter to be identified in this setting, namely the absorption density, is treated as a space-dependent factor within a PDE source term. This model can be considered as an extension of the one in \cite{kaltenbacher_2022}, with the added capability of incorporating the laser intensity function used for excitation. The inverse problem is approached within a Bayesian framework,  where we compute the MAP estimate.  We follow the adjoin-then-discretize approach and derive an explicit expression for the adjoint state by introducing a suitable adjoint problem.   Additionally, in the discretized setting, we compare different choices of priors, some of which are edge-promoting.
\revision{}{Theoretical foundations are given by a well-posedness proof for the forward problem with general right hand side, thus also of the adjoint problem, as well as by proving existence of the MAP estimate as a minimizer.}

\item[(2)] We formulate the optimal experimental design problem in a Bayesian setting, where both parameter and observation space are considered to be infinite-dimensional. Specifically, we consider the $A$-optimality criterion in a setting with Gaussian noise and a Gaussian prior, which results in minimizing the trace of the covariance operator in the posterior distribution. Since directly computing the trace is known to be computationally challenging due to the need to invert a large dense matrix, we introduce an approximation scheme that projects the data misfit Hessian onto finite-dimensional subspaces. We analyze the convergence properties of this projection scheme, establishing conditions under which the solutions of the approximated OED problems converge to the solution of the original infinite-dimensional problem, thus demonstrating the stability and consistency of our approach. Finally, in the discretized setting, we illustrate performance of the projection scheme with a compatible choice of finite-dimensional subspaces.
\revision{}{Also on the OED side we theoretically underpin the proposed methodology by proving existence of an optimal illumination function}
\end{enumerate}

\subsection*{Organization of the paper} The paper is organized as follows.  In Section \ref{sec:preliminaries} we introduce some basic definitions and concepts which will be used throughout the paper.  Section \ref{sec:inverseproblem} is devoted to studying the forward and inverse problems of PAT in a fractional setting, also introducing an infinite-dimensional Bayesian framework for the inverse problem.  Section \ref{sec:OED} focuses on optimizing the laser intensity function for PAT.   In Section \ref{sec:OED_discretized}, we discuss discretization of the infinite-dimensional problem and provide a suitable computational framework.  Finally,  in Section \ref{sec:numerical_results} we present some numerical results in order to illustrate our theory.

\section{Notation and preliminaries}\label{sec:preliminaries}
To begin, we introduce some notations and definitions that will be used throughout the paper.  

\subsection{Fractional derivatives and fractional Sobolev spaces}
We introduce the time-fractional derivative of Caputo-Djrbashian type which is used in our attenuation model.  The Caputo-Djrbashian derivative of order $\alpha \in (0,1)$ of a function $u \in C^1[0,T]$ is defined by
\[
D_t^{\alpha} u := A^{1-\alpha} [u'],
\]
where for $\beta \in (0,1)$ the Abel integral operator $A^{\beta}$ is defined by
\begin{equation}\label{eqn:Abel}
A^{\beta}[v](t) := \frac{1}{\Gamma(\beta)} \int_0^t \frac{v(s)}{(t-s)^{1-\beta}} \, ds = (k_{\beta}*v)(t)
\text{ with }k_{\beta}(t)=\frac{1}{\Gamma(\beta)\, t^{1-\beta}} 
\end{equation}
and $\partial_t^\alpha$ denotes its partial version in case of a function depending on $t$ and $x$.  For details on Caputo-Djrbashian derivative as well as other types of time-fractional derivatives,  for instance the Riemann-Liouville fractional derivative,  we refer to 
\cite{Caputo:1967,Dzharbashyan:1966t,Djrbashian:1993,book_frac,SamkoKilbasMarichev:1993}, and the references therein.

We next introduce Sobolev spaces of fractional order, which are natural functional spaces to study fractional derivatives.  Let $s \in (0,1)$ and $\mathcal{V}$ be a Hilbert space.  The $\mathcal{V}-$valued fractional Sobolev space $H^s((0,T),  \mathcal{V})$ of order $s$ defined in an interval $(0,T)$ is a Hilbert space endowed with the inner product
\begin{align*}
\inner{u, v}_{H^s((0,T),  \mathcal{V})} = &\inner{u, v}_{L^2((0,T),  \mathcal{V})} \\ + &\int_0^T \int_0^T \dfrac{\inner{(u(x) - u(y),(v(x) - v(y)}_{\mathcal{V}}}{|t-s|^{1+2s}} dx dy, \quad u,v  \in H^s((0,T),  \mathcal{V}).
\end{align*}
The space $H_0^s((0,T),V)$ is the closure of $C^{\infty}_c((0,T),  \mathcal{V})$ \footnote{\revision{}{the space of all compactly supported infinitely often differentiable functions}} in $H^s((0,T), \mathcal{V})$ whose dual is $H^{-s}((0,T),\mathcal{V}')$.  In the following,  we are interested in the case $\mathcal{V} = H^r(\Omega)$ for some $r \in \mathbb{R}$.  For more details, 
the reader is referred to \cite[Section 1.9]{Lions_Magenes_1972_I},  \cite[Section 2.4]{Lions_Magenes_1972_II},  \cite{Amann_1997} and references therein.

\subsection{Trace-class and Hilbert-Schmidt operators}
We next introduce some classes of operators that are typically used in the context of OED in an infinite-dimensional setting.  Let $H$ be a separable Hilbert space together with an orthonormal basis $\{e_k\}_{k \in \mathbb{N}}$.  For a compact operator $A: H \to H$,  we say that $A$ is Hilbert--Schmidt if
$$ \norm{A}_{\HS}^2:= \sum_{k=1}^{\infty} (A e_k,  Ae_k)_{H} < \infty,$$
and $A$ is of trace-class if
$$ \tr_{H}{A}:= \sum_{k=1}^{\infty} (A e_k, e_k)_{H} < \infty.$$
In both cases, the sum does not depend on the choice of the basis.  It can be seen that the product of two Hilbert-Schmidt operators has finite trace-class norm.  Therefore,  for two Hilbert-Schmidt operators $A_1,  A_2$,   one can define the Hilbert-Schmidt inner product
\[ (A_1,  A_2)_{\HS} = \tr_{H}(A_1^*A_2) = \sum_{k=1}^{\infty} (A_1 e_k,  A_2 e_k)_{H}. \]
With this inner product,  the space of Hilbert-Schmidt operators on $H$,  denoted by $\HS(H)$, is a Hilbert subspace of the space of bounded linear operators on $H$,  denoted by $L(H)$.  In fact,  $\HS(H)$ is isometrically isomorphic to the tensor product 
$H \otimes H$.
In particular,  an orthonormal basis of this space is $\{ e_k \otimes e_j \}_{ k \in \mathbb{N},  j \in \mathbb{N}}$.

We note that the Hilbert-Schmidt norm and the trace of a compact operator can also be characterized by its eigenvalues with
$$\norm{A}_{\HS}^2 = \sum_{k=1}^{\infty} \lambda_k^2,$$
and
$$\tr_{H}(A) = \sum_{k=1}^{\infty}\lambda_k,$$
where $\lambda_k$ is the $k-$th eigenvalue of $A$, provided these sums converge.  

For details on Hilbert-Schmidt and trace-class operators,  we refer to, for instance, \cite[Chapter VI]{Reed_Simon_1980}.

\section{Setting of the photoacoustic tomography problem}\label{sec:inverseproblem}
\subsection{Forward and adjoint problems}\label{sec:forward_adjoint} 
Recall the time-fractional wave model \eqref{eq:wave_damped} for PAT,  posed in a $C^{1,1}$ smooth bounded domain  $\Omega \subset \R^d$, $d \in \{2, 3\}$.  The solution of \eqref{eq:wave_damped},  which is the acoustic pressure, is denoted by $p = p(t,x)$.  The measurement data are taken on an observation surface $\Sigma$ which is a Lipschitz $(d-1)-$submanifold of $\overline{\Omega}$ over a finite time interval $(0, T)$, cf. \eqref{eq:obs}.  The parameter $a$ in \eqref{eq:wave_damped} to be identified from these indirect measurement is assumed to belong to $L^2(\Omega)$, and the intensity function $i$ to be contained in $ H^1(0,T)$, so that the source term $a(x)i'(t)$ in \eqref{eq:wave_damped} is in $L^2(0,T;L^2(\Omega))$.

In order to study the inverse problem in the fractional setting \eqref{eq:wave_damped}, it is  convenient to consider a general source term, as it describes both the forward and adjoint problem, which are crucial for the reconstruction.  More precisely, we consider the problem
\begin{equation}\label{eq:wave_equation}
\va{ c_0^{-2}\partial^2_t u - \Delta u - b_0\partial_t^\alpha\Delta u&= f&\text{ in } (0,T) \times \Omega,\\
u &= 0 &\text{ on }(0,T)\times \partial \Omega,\\
u(0,\cdot) = \partial_t u(0,\cdot) &= 0 &\text{ on }\Omega,}
\end{equation}
where the source has the more general regularity $f \in H^{s(1-\alpha)}(0,T;H^{-s}(\Omega))$ for some $s \in [0,1]$.

Here, for simplicity of exposition, we are imposing homogeneous Dirichlet boundary conditions on a propagation domain $\Omega$ that we assume to be chosen large enough to avoid any significant impact of spurious reflections at the outer boundary $\partial\Omega$ on the domain of interest $\Omega_0$.  We mention in passing that nonreflecting boundary conditions or perfectly matched layers would allow to use a smaller computational domain \cite{Kaltenbacher_2018} but are beyond the scope of this paper.  Note that our analysis also includes the case of an unbounded domain $\Omega = \R^n$, in which the homogeneous Dirichlet boundary conditions are skipped.

{
As relevant in ultrasound propagation, we take power law frequency dependent attenuation into account by using the Caputo-Wismer-Kelvin model \cite{Caputo:1967,Wismer:2006}.
The time differential operator $\partial_t^\alpha$ in this model is the Djrbashian-Caputo version of the fractional time derivative, which is essential in order to allow for prescribing initial conditions.
For details on fractional differentiation, we refer to, e.g.,
\cite{Dzharbashyan:1966t,Djrbashian:1993,SamkoKilbasMarichev:1993},
see also the tutorial on inverse problems for anomalous diffusion processes
\cite{JinRundell:2015} and the recent book \cite{book_frac}, which also contains well-posedness results for fractionally damped wave equations.}

We begin by studying well-posedness of the problem \eqref{eq:wave_equation}.

\begin{definition}\label{def:solwave}
A function $u$ is a weak solution of \eqref{eq:wave_equation} if
\begin{enumerate}
    \item $u \in L^2(0,T;H^1_0(\Omega))$, 
    $\partial_t^\alpha u\in  L^2(0,T;H_0^1(\Omega))$, 
    and $\partial^2_t u \in L^2(0,T; H^{-1}(\Omega))$,
    \item $u(0,\cdot) = \partial_t u (0,\cdot) = 0$,
    \item for any $v \in H^1_0(\Omega)$ and a.e. $t \in [0,T]$, there holds
\begin{equation}\label{eq:wave_equation_weak}
\dual{c_0^{-2}\partial^2_t u, v}_{H^{-1}(\Omega),H^1_0(\Omega)} + \int_{\Omega} \left( b_0 \partial^{\alpha}_t \nabla u + \nabla u \right)\cdot \nabla v \, dx= \dual{f,v}_{
{H^{-1}(\Omega),H^1_0(\Omega)}}.
\end{equation}
\end{enumerate}
\end{definition}
Here one can actually consider $f \in L^2(0,T;H^{-1}(\Omega)) \supset H^{s(1-\alpha)/2}(0,T;H^{-s}(\Omega))$.  We then have the following well-posedness result for the forward problem.
{
\begin{proposition}\label{prop:wellposed_forward}
For $\alpha\in(0,1]$, $s\in[0,1]$ and every 
$f \in H^{s(1-\alpha)/2}(0,T;H^{-s}(\Omega))$,  there exists a unique weak solution 
of \eqref{eq:wave_equation}, that is, of \eqref{eq:wave_equation_weak}. Furthermore, the solution map 
\begin{equation}\label{eq:S}
S:H^{s(1-\alpha)/2}(0,T;H^{-s}(\Omega))\to\mathbb{X}^s, \quad f \mapsto u \text{ solving \eqref{eq:wave_equation}}
\end{equation}
is linear and bounded, where
\begin{equation}\label{defU}
\begin{aligned}
\mathbb{X}^s:=\,C^1(0,T;L^2(\Omega))&\cap C(0,T;H_0^1(\Omega)) \\
&\cap H^{(1+\alpha)/2}(0,T;H_0^1(\Omega))\cap H^{1-s(1-\alpha)/2}(0,T;H^s(\Omega)).
\end{aligned}
\end{equation}
\end{proposition}

The proof of Proposition~\ref{prop:wellposed_forward} is given in Appendix~\ref{append:proof_wellposed}.

Using Proposition~\ref{prop:wellposed_forward}, one can define the forward PAT operator 
\begin{equation}\label{eq:W}
\mathcal{W}_{\iii}:L^2(\Omega)\to L^2(0,T;L^2(\Sigma)), \quad a\mapsto \Tr_{(0,T) \times \Sigma}(S[a\iii']),
\end{equation}
which maps $a$ into the observation 
where $S$ is the solution operator for \eqref{eq:wave_damped} and $\Tr$ is the Sobolev trace operator.  Hence,  together with additive noise $\eta$,  the mathematical problem of PAT can be formulated as
\begin{equation}\label{eq:parident}
\text{Given } p_{\text{obs}} \in L^2(\Sigma \times (0,T)),\text{ find }a \in L^2(\Omega_0) \text{ such that } \W_{\iii}[a] + \eta = p_\text{obs}.
\end{equation}

By Proposition~\ref{prop:wellposed_forward}, it can be seen that the forward operator in \eqref{eq:parident} is well-defined and continuous, even as a mapping into a smoother space.

\begin{corollary}\label{cor:forward}
For $\iii\in H^{1}(0,T)$, the forward map 
defined by \eqref{eq:W}
is well-defined and continuous, even as a mapping
$\mathcal{W}_{\iii}: L^2(\Omega) \to H^{1-\theta(1-\alpha)/2}(0,T;H^{\theta-1/2}(\Sigma))$ 
for any $\theta\in[0,1]$.

\end{corollary}

\begin{proof} Due to the fact that we impose homogeneous initial and boundary conditions, the maps  
$a\mapsto S[a \iii']$ for fixed $\iii$ as well as
$\iii\mapsto S[a \iii']$ for fixed $a$ are linear and so are their compositions with the trace operator $\Tr_{(0,T)\times\Sigma}$. By Proposition \ref{prop:wellposed_forward}, as well as the trace theorem and interpolation, one has
\begin{equation}\label{eq:estimate_1}
\begin{aligned}
&\norm{\mathcal{W}_{\iii}[a]}^2_{H^{1-\theta(1-\alpha)/2-(1-\theta)s(1-\alpha)/2}(0,T;H^{\theta-1/2+(1-\theta)s}(\Sigma))} 
\\&\lesssim
\|a \iii'\|_{H^{s(1-\alpha)/2}(0,T;H^{-s}(\Omega))}
= \|a\|_{H^{-s}(\Omega)} \|\iii'\|_{H^{s(1-\alpha)/2}(0,T)}
\end{aligned}
\end{equation}
for any $\theta\in[0,1]$.
We choose $s = 0$ in \eqref{eq:estimate_1} 
to conclude the claimed result.
\end{proof}


By compactness of the embedding $H^{1-\theta(1-\alpha)/2}(0,T;H^{\theta-1/2}(\Sigma))\to L^2(0,T;L^2(\Sigma))$ for $\theta\in(\frac12,1]$, $\mathcal{W}_{\iii}$ is compact as a function from $L^2(\Omega)$ into $L^2(0,T;L^2(\Sigma))$.
\begin{corollary}\label{cor:compact} \text{}
\begin{enumerate}
\item[(1)] For every fixed $a \in L^2(\Omega)$,  the map $\iii \mapsto \W_{\iii}[a]$ from $H^1(0,T)$ to $L^2(0,T;L^2(\Sigma))$ is compact.
\item[(2)] For every fixed $\iii \in H^1(0,T)$,  the map $a \mapsto \W_{\iii}[a]$ from $L^2(\Omega)$ to $L^2(0,T;L^2(\Sigma))$ is compact.
\end{enumerate}
\end{corollary}


\begin{remark}[Injectivity of $\mathcal{W}_{\iii}$] Unique identifiability of $a$, that is, injectivity of the forward operator, follows under certain geometric assumptions on $\Sigma$ and the excitation $\iii$.  For the fractionally damped case,  see e.g.  \cite[Remark 11.3]{kaltenbacher2023}.
\end{remark}

In order to study the inverse problem of identifying $a$ in \eqref{eq:wave_damped},  we make use of the adjoint operator $\mathcal{W}^*_{\iii}$ of $\mathcal{W}_{\iii}$.  Since $\mathcal{W}_{\iii}: L^2(\Omega) \to L^2(0,T;L^2(\Sigma))$ is linear and bounded, its adjoint $\mathcal{W}^*_{\iii}: L^2(0,T;L^2(\Sigma)) \to L^2(\Omega)$ is also bounded.  To derive an explicit expression for $\mathcal{W}^*_{\iii}$,  we denote by $[\,\cdot\,]_{\Sigma}$ the jump of a function across the observation surface $\Sigma$, that is
$$[q]:= q_+|_{\Sigma} - q_{-}|_{\Sigma}.$$

Assume that $g$ 
is sufficiently smooth, more precisely {$g\in H^{s(1-\alpha)/2}(0,T;L^2(\Sigma))$}. We consider the problem
\begin{equation}\label{eq:wave_equation_adjoint}
\va{ c_0^{-2}\partial^2_t q 
+{
b_0 \widetilde{\partial_t^{\alpha}} \Delta q 
}
- \Delta q &= 0 &\text{ in } (0,T) \times \Omega\setminus \Sigma,\\
q &= 0 &\text{ on }(0,T)\times \partial \Omega,\\
 \text{} [q] = 0,  \left[-b_0 \widetilde{\partial_t^{\alpha}}\partial_\nu q+\partial_\nu q\right]  &= g &\text{ on } [0,T] \times \Sigma\\
 q (T,\cdot) = \partial_t q(T,\cdot) &= 0 &\text{ in }\Omega.
}
\end{equation}
where $\nu$ is the outward pointing normal on $\Sigma$ (considered as part of the boundary of a subdomain $\Omega_\Sigma$ of $\Omega$) and 
\begin{equation}\label{eq:right_fractionalderivative}
(\widetilde{\partial_t^{\alpha}}w)(s) := \frac{1}{\Gamma(1-\alpha)} \int_s^T (t-s)^{-\alpha} \partial_t w(t) \, dt.
\end{equation}
In case $w$ is not space dependent, we write $\widetilde{D_t^{\alpha}}w$ for $\widetilde{\partial_t^{\alpha}}w$. The adjoint problem \eqref{eq:wave_equation_adjoint} can be derived by noticing that we have the following integration by parts formula:

\begin{lemma}For every $u,v \in W^{1,1}(0,T)$ and $\widetilde{I}^{\alpha}v := \frac{1}{\Gamma (1-\alpha)} \int_0^T t^{-\alpha} v(t)\,dt$ 
\begin{equation}
\int_0^T (D_t^{\alpha} u)(t) v(t) dt  = - \int_0^T u (t) (\widetilde{D^{\alpha}_t} v)(t)dt - u(0) \widetilde{I}^{\alpha}v.
\end{equation}
\end{lemma}
The proof is given in \cite{kaltenbacher_2022}. Using this identity, one can see that the weak form of \eqref{eq:wave_equation_adjoint} is given by 
\begin{equation}\label{eq:adjoint_formula}
\left\langle c_0^{-2} \partial_t^2 q, v\right\rangle_{H^{-1}(\Omega),H_0^1(\Omega)} + \int_{\Omega} \inner{-b_0 \widetilde{\partial_t^{\alpha}} \nabla q +\nabla q} \cdot \nabla v d x= -\int_{\Sigma} g v d S.
\end{equation}
for a.e. $t \in [0,T]$
and via the identity 
\[
(\widetilde{\partial_t^{\alpha}}w)(T-t) 
= -(\partial_t^\alpha \overline{w})(t)
\mbox{ for }\overline{w}(t):=w(T-t)
\]
can be written as 
\eqref{eq:wave_equation_weak} for the time-flipped adjoint state $u(t)=\overline{q}(t)=q(T-t)$ with 
$$
\langle f(t), v\rangle:=-\int_{\Sigma} g(T-t,x) v(x) \d S(x). 
$$

To prove the existence of a weak solution to the adjoint equation \eqref{eq:wave_equation_adjoint}, we can apply Proposition~\ref{prop:wellposed_forward}, using the fact that the above defined right hand side $f$ lies in $H^{s(1-\alpha)/2}(0,T;H^{-s}(\Omega))$, provided  $s\in(\frac12,1)$ and $g\in H^{s(1-\alpha)/2}(0,T;L^2(\Sigma))$ due to the Trace Theorem.

\begin{corollary}
Assume that $g \in H^{s(1-\alpha)/2}(0,T;L^2(\Sigma))$ for some $s\in(1/2,1]$. Then \eqref{eq:wave_equation_adjoint} admits a solution $q\in \mathbb{X}^s$ as defined in \eqref{defU}.
\end{corollary}

We conclude that in case $g \in L^2(0,T;L^2(\Sigma))$ is sufficiently smooth, the value of the adjoint operator $\mathcal{W}_{\iii}^*[g]$ can be computed via the adjoint problem.
\begin{proposition}\label{prop:adjoint_operator} 
For any  $g \in H^{s(1-\alpha)/2}(0,T;L^2(\Sigma))$, $s\in(1/2,1]$, the Banach (and due to the fact that on $L^2(\Omega)$ the Riesz representation is given by the identity also Hilbert) space adjoint 
$\mathcal{W}^*_{\iii}:L^2(0,T;L^2(\Sigma))\to L^2(\Omega)$ of $\mathcal{W}_{\iii}:L^2(\Omega)\to L^2(0,T;L^2(\Sigma))$ is given by  
\begin{equation}\label{eq:adjoint_operator}
\mathcal{W}^*_{\iii}[g] = \int_0^T q(t,\cdot)\iii^{\prime}(t) \d t,
\end{equation}
where $q$ solves \eqref{eq:adjoint_formula}.
\end{proposition}

\begin{proof}In the weak form of the equation for $\overline{q}$
$$
\dual{c_0^{-2}\partial^2_t u, v}_{H^{-1}(\Omega),H^1_0(\Omega)} + \int_{\Omega} \left( b_0 \partial^{\alpha}_t \nabla u + \nabla u \right)\cdot \nabla v \, dx= 
-\int_{\Sigma} g(T-t) v \d S
$$
choose $v=p(t)$ to be the solution of \eqref{eq:wave_damped}. 
Integrating over the time interval and using the transposition identities
\begin{equation} \label{integbyparts}
\begin{aligned}
\int_0^T w'(t) z(T-t)\,dt
=\int_0^T w(t) z'(T-t)\,dt + w(T)z(0) - w(0)z(T),\\
\quad w,z\in W^{1,1}(0,T),
\end{aligned}
\end{equation}
\begin{equation} \label{integbyparts1}
\begin{aligned}
\int_0^T (k*w)(t)z(T-t)\,dt = \int_0^T w(t) (k*z)(T-t)\,dt,\\
\quad k\in L^1(0,T), \quad w,z\in L^2(0,T)
\end{aligned}
\end{equation}
see, e.g., \cite{frac_TUM},
we have
\begin{equation}
\int_0^T \int_{\Omega} q(t,x) a(x) \iii^{\prime}(t) \d x \d t  = \int_0^T \int_{\Sigma} g\, p\, dS\, dt = \inner{ g, \mathcal{W}_{\iii}[a]}_{L^2(0,T;L^2(\Sigma))}.
\end{equation}
This implies the adjoint $\W^*_{\iii}$ is given by \eqref{eq:adjoint_operator} for every $g \in H^{s(1-\alpha)/2}(0,T;L^2(\Sigma))$. By noticing that this space is dense in $L^2(0,T;L^2(\Sigma))$, we have the desired result.
\end{proof}

As a consequence of Proposition~\ref{prop:adjoint_operator}, in order to apply $\W^*_{\iii}$, we have to approximate the given data $g$ by $\tilde{g}\in H^{s(1-\alpha)/2}(0,T;L^2(\Sigma))$ with some fixed $s\in(\frac12,1)$ (see, e.g., \cite[Section 8.2.3.2.]{book_frac} which is computationally relatively easy in the one-dimensional time direction).
Note that application of $\mathcal{W}_{\iii}^*$ to some $\mathcal{W}_{\iii}[a]$ is justified by Proposition~\ref{prop:adjoint_operator} without any need for smoothing, due to the fact that according to Proposition~\ref{prop:wellposed_forward}, $\mathcal{W}_{\iii}[a]$ even lies in $H^{(1+\alpha)/2}(0,T;L^2(\Sigma))$.

\subsection{Bayesian inverse problem} Following \cite{stuart_2010, stuart_2017},  in this section we study the inverse problem of PAT in the Bayesian setting.  This inverse problem can be written in the abstract setting as
\begin{equation}\label{eq:abstract_inverse_problem}
p_{\text{obs}} = \mathcal{W}_{\iii}[a] + \eta, 
\end{equation}
where we recall that $\mathcal{W}_{\iii}$ is the (experiment design dependent) forward map given in Corollary~\ref{cor:forward},  $p_{\text{obs}} \in L^2(0,T;L^2(\Sigma))$ is the measurement data and $\eta$ is the additive noise.  We note that in our setting, both the parameter and the observation are infinite-dimensional. The parameter and observation spaces are given by
\begin{equation}\label{XY}
X := L^2(\Omega), \qquad Y:= L^2(0,T;L^2(\Sigma)).  
\end{equation}
In addition, by Corollary~\ref{cor:forward} the forward map $\mathcal{W}_{\iii}$ is compact.  Hence,  the problem is ill-posed and regularization must be applied. We do so by adopting a Bayesian approach that incorporates prior information and provides a means of quantifying uncertainty due to observational noise. 

In \eqref{eq:abstract_inverse_problem},  we assume that $a$ is normally distributed,  i.e., $a \sim \mu_0 = \mathcal{N}(a_0, \Gamma_{\pr})$,  where $\Gamma_{\pr}$ is a self-adjoint, \revision{poisitive-definite}{positive-definite} bounded linear operator of trace class.  The noise model is also Gaussian,  $\eta \sim \mathcal{N}(0,  \Gamma_{\text{noise}})$,  where $\Gamma_{\text{noise}}$ is a self-adjoint positive-definite bounded linear operator, but \textit{not} necessarily of trace class,  which allows us to include the case of  observations under white noise.  For the sake of convenience,   we choose $\Gamma_{\text{noise}} = \sigma^2 I$ where $\sigma^2$ is fixed.  (A general case where $\sigma^2 = \sigma_n^2 \to 0$ as $n \to \infty$ is considered in \cite{Agapiou_Larsson_Stuart_2013,  Kahle_Lam_Latz_Ullmann_2019}.)  Hence, the posterior distribution for $a$ given $p_{\text{obs}}$,  i.e.  the distribution $a|p_{\text{obs}}$ is also Gaussian,  i.e.  $a| p_{\text{obs}} \sim \mathcal{N}(a_{\text{MAP}},  \Gamma_{\post})$, with the posterior mean coinciding with the map estimator, see \eqref{eq:map_estimator} below.

In \revision{}{the case that} the observation space is finite-dimensional, the data-likelihood is given by \cite{Alexanderian_2014, stuart_2010} 
\[ \pi_{\text{like}}(p_{\text{obs}}|a) \propto \exp \left[ -\dfrac{1}{2}(\W_{\iii} a-p_{\text{obs}})^{\revision{\top}{*}}\Gamma^{-1}_{\text{noise}}(\W_{\iii} a - p_{\text{obs}}) \right], \]
which implies that the posterior covariance operator is given by 
\[
\Gamma_{\post} = (\W_{\iii}^*\Gamma^{-1}_{\text{noise}}\W_{\iii} + \Gamma_{\pr}^{-1})^{-1},
\]
see for instance \cite[Theorem 2.4]{stuart_2010}.  Following~\cite{walter_thesis_2019}},  we justify an analogous formula when the observational space is infinite-dimensional provided that the prior covariance operator belongs to a suitable space.  More precisely,  we assume that:

\begin{assumption}\label{assum:prior} The prior distribution is $\mu_{\pr} = \mathcal{N}(a_0,  \Gamma_{\pr})$,  where $a_0\in L^2(\Omega)$,  $\Gamma_{\pr}$ is of trace class in $L^2(\Omega)$.
\end{assumption}

Under this assumption, we can prove the following:

\begin{proposition}\label{prop:traceclass_posterior} Assume that $\Gamma_{\text{noise}} = \sigma^2 I$ and $\Gamma_{\pr}$ satisfies Assumption~\ref{assum:prior}.  Then the posterior covariance operator $\Gamma_{\post}: L^2(\Omega) \to L^2(\Omega)$ is given by
\begin{equation}
\Gamma_{\post}  = \Gamma_{\post}(\iii) = \left(\sigma^{-2}\W_{\iii}^* \W_{\iii} + \Gamma_{\pr}^{-1}\right)^{-1}.
\end{equation}
Moreover, the MAP estimator is given by
\begin{equation}\label{eq:map_estimator}
a_{\text{MAP}} = \Gamma_{\post}\left(
\sigma^{-2}\W_{\iii}^*
p_{\text{obs}} + \Gamma_{\pr}^{-1}a_0 \right).
\end{equation}
In addition,  $\Gamma_{\post}$ is positive and of trace class in $L^2(\Omega)$.
\end{proposition}

\begin{proof}
Note that due to Assumption~\ref{assum:prior}, $\Gamma_{\pr}$ is bounded as a mapping from $L^2(\Omega)$ into itself. Thus the image space $\Hone=\Gamma_{\text{pr}}^{-1/2}(L^2(\Omega))$, equipped with the norm $ \norm{v}_{\Hone}=\norm{\Gamma^{-1/2}_{\pr} v}_{L^2(\Omega)}$ is continuously embedded in $L^2(\Omega)$.

We adapt the proof in \cite[Proposition 5.9]{walter_thesis_2019}.  
First we show that $\Gamma_{\post}$ is a bounded map from $L^2(\Omega)$ to $L^2(\Omega)$ (in fact into $\Hone$).  Consider the bilinear form $\mathcal{B}: \Hone \times  \Hone \to \R$ given by
\begin{align*}
\mathcal{B}(u,v) = \sigma^{-2}(\W_{\iii}u,  \W_{\iii} v)_{L^2(0,T;L^2(\Sigma))} + (\Gamma^{-1/2}_{\pr} u,  \Gamma^{-1/2}_{\pr} v)_{L^2(\Omega)},  \quad u,v  \in \Hone.
\end{align*}
It can be seen that $\mathcal{B}$ is bilinear,  symmetric, continuous and coercive 
\begin{align*}
\left|\mathcal{B}(u,v) \right| \lesssim \norm{u}_{\Hone} \norm{v}_{\Hone},
\qquad
\mathcal{B}(u,u) 
\gtrsim \norm{u}_{\Hone}^2.
\end{align*}
for all $u,v\in\Hone$.
Hence by the Lax-Milgram Theorem,  for every $a \in L^2(\Omega)$,  there exists a unique $q_a \in \Hone \subset L^2(\Omega)$ such that
\begin{align*}
\mathcal{B}(q_a,   v) = \inner{a, v}_{L^2(\Omega)},\quad \forall v \in \Hone,
\end{align*}
and therefore $q_a = (\sigma^{-2}\revision{\W_{\iii^*}}{\W_{\iii}^*}\W_{\iii} + \Gamma_{\text{pr}}^{-1})^{-1} a$.  Following the computation in \cite[Example 6.23.]{stuart_2010},  we know that $(\sigma^{-2}\revision{\W_{\iii^*}}{\W_{\iii}}\W_{\iii} + \Gamma_{\text{pr}}^{-1})^{-1}$ coincides with the covariance operator of the posterior distribution of $a|p_{\text{obs}}$.  
The fact that the operator is of trace class on $L^2(\Omega)$ follows by noting that for every $e_k \in L^2(\Omega)$,  one has
\[(\sigma^{-2}\revision{\W_{\iii^*}}{\W_{\iii}^*}\W_{\iii} + \Gamma_{\text{pr}}^{-1})^{-1} e_k,  e_k) _{L^2(\Omega)} \le (\Gamma_{\pr}e_k,  e_k)_{\revision{H}{L^2(\Omega)}},\]
since $\revision{\W_{\iii^*}}{\W_{\iii}^*}\W_{\iii}$ is nonnegative definite; hence
\begin{align*}
\Tr_{L^2(\Omega)}(\sigma^{-2}\revision{\W_{\iii^*}}{\W_{\iii}^*}\W_{\iii} + \Gamma_{\text{pr}}^{-1})^{-1} 
&= \sum_{k=1}^{\infty} \left((\sigma^{-2}\revision{\W_{\iii^*}}{\W_{\iii}^*}\W_{\iii} + \Gamma_{\text{pr}}^{-1})^{-1} e_k,  e_k\right)_{L^2(\Omega)} \\
&\le \sum_{k=1}^{\infty} (\Gamma_{\pr}e_k,  e_k)_{L^2(\Omega)} = \Tr_{L^2(\Omega)} \Gamma_{\text{pr}} < \infty.
\end{align*}
\end{proof}
We note that the posterior covariance operator $\Gamma_{\post}$ can also be written as
\[
\Gamma_{\post} = \Gamma^{1/2}_{\pr} \left(\sigma^{-2}\Gamma^{1/2}_{\pr} \revision{W}{\W_{\iii}\W_{\iii}^*}\Gamma^{1/2}_{\pr} + I \right)^{-1}\Gamma^{1/2}_{\pr},
\]
which avoids the inverse $\Gamma_{\pr}^{-1}$.  

As we have seen in the proof of Proposition~\ref{prop:traceclass_posterior}, the MAP estimator can be computed by solving the Tikhonov-type regularization problem
\begin{equation}\label{eq:tikhonov}
\min_{a \in L^2(\Omega)} J(a) = \dfrac{1}{2\sigma^2}\norm{(\W_{\iii}[a] - p_{\text{obs}})}^2_{L^2(0,T;L^2(\Sigma))} + \dfrac{1}{2}\norm{\Gamma_{\pr}^{-\frac{1}{2}}(a - a_0)}^2_{L^2(\Omega)}.
\end{equation}
Details on the computation of the MAP estimate will be given in Section~\ref{sec:numerical_results}.

\section{Optimal design of illumination function}\label{sec:OED}
We are interested in optimizing the illuminating function $\iii$ in PAT, relying on the dependence of  the forward operator in the PAT reconstruction problem on $\iii$.  
In doing so, we incorporate possible additional constraints on $\iii$ by minimizing over a weakly compact subset $\mathcal{D}$ of $H^1(0,T)$, e.g., a ball $\mathcal{B}_R^{H^1(0,T)}(\iii_0)$ in $H^1(0,T)$ centered at some a priori guess $\iii_0$ with sufficiently large radius $R$.  

\subsection{Optimality criterion}
Recall that the posterior covariance operator $\Gamma_{\post}$,  which is considered as a function of $\iii$, is given by
\[ \Gamma_{\post} = \Gamma_{\post}(\iii) = \left(\sigma^{-2}\mathcal{W}^*_{\iii}\mathcal{W}_{\iii} + \Gamma_{\pr}^{-1}\right)^{-1}.\]

In order to optimize the laser excitation function,  we consider the well-known $A$-optimality criterion for Bayesian optimal design of experiments \cite{Alexanderian_2016,  Haber_2008}
\begin{equation}\label{eq:A_optimal}
\va{ &\min_{\iii} \varphi(\iii) = \Tr_{L^2(\Omega)} \left[\Gamma_{\post}(\iii)\right] \\
&\text{subject to } \iii \in \mathcal{D}}
\end{equation}
where $\mathcal{D}$ is a weakly compact subset of $H^1(0,T)$.  Here,  the objective functional in \eqref{eq:A_optimal} is finite for every $\iii\in H^1(0,T)$ by Proposition~\ref{prop:traceclass_posterior}. The $A$-optimality criterion in \eqref{eq:A_optimal} has been extensively studied in the existing literature and is well-understood for finite-dimensional observation space.  In the setting with infinite-dimensional observation space which is considered in this work,  this criterion is justified by the formula for the pointwise variance of the posterior distribution  \cite{walter_thesis_2019,  Alexanderian_2021}: 
\begin{align*}
&\int_{\Omega} 
\var[a]
\, dx 
= \int_{\Omega} \int_{\mathcal{H}} |a(\omega,x) - a_{\text{MAP}}(x)|^2d \mathbb{P}(\omega) dx  \\
&= \int_{\mathcal{H}} \int_{\Omega} |a(\omega,x) - a_{\text{MAP}}(x)|^2 dx d \mathbb{P}(\omega) 
= \int_{\mathcal{H}} \norm{a(\omega,\cdot) - a_{\text{MAP}}}^2_{L^2(\Omega)} d \mathbb{P}(\omega)\\
&= \int_{\mathcal{H}} \norm{a - a_{\text{MAP}}}^2_{L^2(\Omega)} d\mu_{\post} = \tr(\Gamma_{\post}).
\end{align*}

We next prove the existence of a minimizer of \eqref{eq:A_optimal}.  For this purpose,  since we know that the forward map is compact,  we need to quantify this compactness (that is,  the ill-posedness of the inverse problem) as follows:
\begin{assumption}\label{ass:reg} 
There exists $\gamma> d/4$ such that the forward operator $\W_{\iii}:  L^2(\Omega) \to H^{\gamma}((0,T) \times \Sigma)$
is bounded uniformly with respect to $\iii\in H^1(0,T)$, that is, there exists $C_{\W}>0$ such that for all $a\in L^2(\Omega)$, $\iii\in H^1(0,T)$
\begin{equation}\label{eq:gamma}
\|\W_{\iii}[a]\|_{H^{\gamma}((0,T) \times \Sigma)}\leq C_{\W} \|a\|_{L^2(\Omega)}\, \|\iii\|_{H^1(0,T)}.
\end{equation}
\end{assumption}

Under the given assumption, we are able to prove the existence of minimizers of \eqref{eq:A_optimal}. 

\begin{proposition}\label{prop:existence_minimizer}
Under Assumptions~\ref{assum:prior}, \ref{ass:reg}, problem \eqref{eq:A_optimal} admits a minimizer.
\end{proposition}

In order to prove the existence of a minimizer of \eqref{eq:A_optimal}, we first prove that the normal operator $\W_{\iii}^*\W_{\iii}$ belongs to a suitable function space with operator norm uniformly bounded in $\iii$. More specifically, we denote by $s_k(\W_{\iii})$ the $k-$th singular value of $\W_{\iii}$.  By Assumption \ref{ass:reg} together with Lemmata \ref{lem:singular} and \ref{lem:embedding}, we have the following characterization of $s_k(\W_{\iii})$.

\begin{lemma}\label{lem:singular_W} 
Under Assumptions~\ref{assum:prior}, \ref{ass:reg}, for every $\iii \in H^1(0,T)$,  there holds
\begin{equation}\label{eq:singular_identity}
s_k(\W_{\iii}) \le C_{\W}\norm{\iii}_{H^1(0,T)} k^{-\gamma/d},  \quad k \in \mathbb{N}.
\end{equation}
Hence,  $\W_{\iii}^*\W_{\iii}$ is a Hilbert-Schmidt operator.
\end{lemma}

\begin{proof}One can write $\W_{\iii} = j \circ \widetilde{\W}_{\iii}$ where $j: H^\gamma((0,T)\times \Sigma) \to L^2( (0,T) \times \Sigma)$ is the natural embedding.  By Lemma \ref{lem:embedding} and \ref{lem:singular},  we have 
\begin{align*}
s_k(\W_{\iii}) = s_k(j \circ \widetilde{\W}_{\iii}) \le \|\widetilde{\W}_{\iii} \| s_k(j) \le C_{\W} \norm{\iii}_{H^1(0,T)} k^{-\gamma/d}.
\end{align*}
This implies
\begin{equation}\label{eq:HS_norm_W}
\norm{\W_{\iii}^*\W_{\iii}}_{\HS}^2 =  \sum_{k = 1}^{\infty} |\lambda_k(\W_{\iii}^*\W_{\iii})|^2 =  \sum_{k = 1}^{\infty} s_k(\W_{\iii})^{4} \le C_{\W}^4 \norm{\iii}_{H^1(0,T)}^4 \sum_{k=1}^{\infty} \ k^{-4\gamma/d} < \infty,
\end{equation}
since by Assumption \ref{ass:reg} we know that $4\gamma/d > 1$.  Hence,  this implies that $\W_{\iii}^*\W_{\iii}$ is a Hilbert-Schmidt operator.  The proof is complete.
\end{proof}

\begin{remark}[Singular value decay]\label{rem:singvals}
In the photoacoustic tomography problem in homogeneous media without attenuation, and if the observation set $\Sigma$ is a sphere, Assumption~\ref{ass:reg} holds true for $d = 3$ \cite{Finch_Rakesh_2005} and $d = 2$ \cite{Finch_Haltmeier_Rakesh_2007}. 

The asymptotic behaviour of singular values of the forward map $\W_{\iii}$ in photoacoustic tomography formulated as an inverse initial value problem (that is, with $\iii=\delta$) with fractional attenuation has been considered in \cite{Elbau_Scherzer_Shi_2017}.  In particular,  in both strongly damped and weakly damped setting,  the normal operator $\W_{\iii}^*\W_{\iii}$ is also known to be Hilbert-Schmidt.  
In the terminology of \cite{Elbau_Scherzer_Shi_2017}, the time-fractional damping model used here is causal and strongly attenuating, cf. \cite{Elbau_Shi_Scherzer_2024}; thus the singular values even decay exponentially, cf. \cite[Corollary 5.2]{Elbau_Scherzer_Shi_2017}.

However, in our setting such decay rates have not been established yet. Note that the estimate from Corollary~\ref{cor:forward} only implies \eqref{eq:gamma} with $\gamma=\frac12$.
A way to enforce the decay \eqref{eq:singular_identity} for a general (just bounded) forward operator $\W_{\iii}:H^\ell(\Omega)\to Y$ is to strengthen the preimage space to a higher order Sobolev space, thus considering $\W_{\iii}:H^{\ell+\delta}(\Omega)\to Y$ and making use of the decay of the singular values of the embedding operator $H^{\ell+\delta}(\Omega)\to H^\ell(\Omega)$ (cf. the analog of Lemma~\ref{lem:embedding} on $\Omega$ in place of $(0,T)\times\Sigma$).
In our case (with $\ell=0$), this would require $\delta+\frac12>\frac{d}{4}$.
\end{remark}

Hence,  one can derive the continuity of the map $\iii \mapsto \mathcal{W}^*_{\iii} \mathcal{W}_{\iii}$.
\begin{lemma}\label{prop:continuity} The map $\iii \mapsto \mathcal{W}^*_{\iii} \mathcal{W}_{\iii}$ is weak-to-weak continuous and strong-to-strong continuous from $H^1(0,T)$ to $\HS(L^2(\Omega))$.
\end{lemma}

\begin{proof}We first prove the weak continuity.  Consider a sequence $\{\iii_m\}_{m \in \mathbb{N}} \subset H^1(0,T)$ that converges weakly to $\iii$.  Then $\{\iii_m\}_{m \in \mathbb{N}}$ is bounded in $H^1(0,T)$. We apply Lemma~\ref{lem:singular} to have
\begin{align*}\label{eq:HS_norm_bounded}
\norm{\W_{\iii_m}^*\W_{\iii_m}}_{\HS}^2 = \sum_{k=1}^{\infty} |s_k(\W_{\iii_m})|^4 
&\le C_{\W}^4 \norm{\iii_m}^4_{H^1(0,T)} \sum_{k=1}^{\infty} k^{-4\gamma/d},
\end{align*}
which means that the sequence $\{\W_{\iii_m}^*\W_{\iii_m}\}_{m \in \mathbb{N}}$ is uniformly bounded in the Hilbert-Schmidt norm.  In addition,  for every basis function $e_k \otimes e_j \in \HS(L^2(\Omega))$ (see Section \ref{sec:preliminaries}),  one has
\begin{equation}
\begin{aligned}
(\W_{\iii_m}^*\W_{\iii_m},  e_k \otimes e_j)_{\HS} 
&= (\W_{\iii_m}^*\W_{\iii_m} e_k,  e_j)_{L^2(\Omega)} 
= (\W_{\iii_m} e_k,  \W_{\iii_m} e_j)_{L^2( (0,T) \times \Sigma)} \\
&\to (\W_{\iii} e_k,  \W_{\iii} e_j)_{L^2( (0,T) \times \Sigma)} = (\W_{\iii}^*\W_{\iii},  e_k \otimes e_j)_{\HS} 
\end{aligned}
\end{equation}
due to the compactness of the map $\iii \mapsto \W_{\iii} e$ for every $e \in L^2(\Omega)$ by Corollary~\ref{cor:compact}.  Combining this with the uniform boundedness of $\{\W_{\iii_m}^*\W_{\iii_m}\}_{n \in \mathbb{N}}$,  we can use the Dominated Convergence Theorem to conclude that
\begin{equation}\label{eq:weakconvergence}
\W_{\iii_m}^*\W_{\iii_m} \rightharpoonup \W_{\iii}^*\W_{\iii} \text{ in }\HS(L^2(\Omega)).
\end{equation}
Now we prove the strong continuity by assuming that $\iii_m \to \iii$ in $H^1(0,T)$.   Then it is sufficient to show that
\begin{align*}
\norm{\W_{\iii_m}^*\W_{\iii_m}}_{\HS} \rightarrow \norm{\W_{\iii}^*\W_{\iii}}_{\HS}.
\end{align*}
Indeed,  by the continuity of the map $\iii \mapsto \W_{\iii}$,  it can be seen that $\W_{\iii_m} \to \W_{\iii}$ in the operator norm and therefore
\begin{align*}
\W_{\iii_m}^*\W_{\iii_m} \to \W_{\iii}^*\W_{\iii} \text{ in }L(L^2(\Omega)).
\end{align*}
In particular,  by \cite[Section 11.1.1]{Pietsch_1980},  one has $s_k(\W_{\iii_m}) \to s_k(\W_{\iii})$ for every $k \in \mathbb{N}$.  Together with the fact that $s_k(\W_{\iii_m}) \le C_{\W} \norm{\iii_m}_{H^1(0,T)}  k^{-\gamma/d}$ for all $k \in \mathbb{N}$, and boundedness of $\norm{\iii_m}$, by \revision{}{the} dominated convergence we conclude that
\begin{align*}
\norm{\W_{\iii_m}^*\W_{\iii_m}}_{\HS}^2 = \sum_{k=1}^{\infty} s_k(\W_{\revision{\iii_n}{\iii_m}})^4 \to \sum_{k=1}^{\infty} s_k(\W_{\iii})^4  =  \norm{\W_{\iii}^*\W_{\iii}}_{\HS}^2,
\end{align*}
which implies \revision{}{the} convergence of the norm.  Combining with the weak convergence in \eqref{eq:weakconvergence},  we conclude that 
$$ \W_{\revision{\iii_n}{\iii_m}}^*\W_{\revision{\iii_n}{\iii_m}} \to \W_{\iii}^*\W_{\iii} \text{ in }\HS(L^2(\Omega)),$$
\revision{}{which implies the strong continuity of the map $i \mapsto \W_i$. The proof is complete.}
\end{proof}

\revision{}{We are now ready to prove the main result of this section, namely Proposition~\ref{prop:existence_minimizer}.}

\begin{proof}[\sc Proof of Proposition~\ref{prop:existence_minimizer}] Consider a minimizing sequence $\{ \iii_m \}_{m \in \mathbb{N}} \subset H^1(0,T)$ with $\{\iii_m\}_{m \in \mathbb{N}} \subset \mathcal{D}$,  i.e.
\begin{align*}
\overline{m} = \lim_{m \to \infty} \varphi(\iii_m).
\end{align*}
Since $\mathcal{D}$ is weakly compact, the sequence $\{ \iii_m \}_{m \in \mathbb{N}}$ admits a weak limit point $\iii^* \in \mathcal{D}$.  Also, notice that the map
\begin{equation}\label{eq:trace_continuity}
g: \HS(L^2(\Omega))\to\mathbb{R}, \quad A \mapsto g(A) := \Tr_{L^2(\Omega)}\left[(A+ \Gamma_{\pr}^{-1})^{-1}\right]
\end{equation}
is continuous and convex which implies that it is weakly lower-semicontinuous.  Hence,  since $\W_{\iii_m}^*\W_{\iii_m} \rightharpoonup \W_{\iii}^*\W_{\iii}$ in $\HS(L^2(\Omega))$ by Proposition~\ref{prop:continuity},  we have
\begin{align*}
\varphi(\iii^*) = g(\W_{\iii^*}^*\W_{\iii^*}) \le \liminf_{m \to \infty} g(\W_{\iii_m}^*\W_{\iii_m}) = \liminf_{m \to \infty} \varphi(\iii_m) = \overline{m},
\end{align*}
which implies that $\iii^*$ is a minimizer of \eqref{eq:A_optimal}. 
\end{proof}

In contrast to the $A$-optimality functional considered in \cite{Alexanderian_2016} which optimizes the sensor placement setup,  the functional $\varphi$ is non-convex here, due to the nonlinearity of the map $\iii \mapsto \Gamma_{\post}(\iii)$ (see also the following section). Consequently, finding a global solution to \eqref{eq:A_optimal} appears to be practically infeasible. Nevertheless, an improvement of an initial design $\iii_0$ is often considered sufficient.

\subsection{Projection scheme}\label{sec:projection} As we have pointed out previously, the computation of the trace in \eqref{eq:A_optimal} together with its derivative with respect to $\iii$ after discretization requires the inversion of a large dense matrix, or at least the computation of its action on vectors. This is computationally challenging and makes the OED problem intractable.  Therefore,  an approximation scheme for \eqref{eq:A_optimal} should be employed.  In this subsection, we examine an  approach via projection onto finite-dimensional subspaces.  To be more precise,  we consider a sequence of finite-dimensional subspace of $L^2(\Omega)$,  denoted by $\{X_k\}_{k \in \mathbb{N}}$ with the property that 
\[\dim X_k = k,  X_k \subset X_{k+1} \text{ for all } k \in \mathbb{N},  \quad \text{ and }\overline{ \cup_{k = 1}^{\infty} X_k } = L^2(\Omega). \]
We denote by $P_N: L^2(\Omega) \to X_N \subset L^2(\Omega)$ the orthogonal projection onto $X_N$. 

The following result shows the existence of an orthonormal basis of each $X_k$, that extends to an orthonormal basis of $L^2(\Omega)$. It can be constructed using the Gram-Schmidt orthonormalization process and we omit the detailed proof for brevity.
\begin{lemma}\label{lem:orthonormalbasis} There exists an orthonormal basis $\{e_k\}_{k \in \mathbb{N}}$ of $L^2(\Omega)$ such that $e_k \in X_k$ for every $k \in \mathbb{N}$ and $X_k = \spann \{ e_1, \ldots,  e_k\}$.
\end{lemma}

With this sequence of orthogonal subspaces,  we consider the \textit{projected data misfit Hessian} corresponding to $X_N$,  $N \in \mathbb{N}$,  which is defined by
\begin{equation}\label{eq:proj_misfit}
H^N_{\text{misfit}} := \sigma^{-2} P_N\mathcal{W}_{\iii}^* \mathcal{W}_{\iii} P_N
\end{equation}
together with the \text{projected posterior covariance operator}
\begin{equation}\label{eq:approximated_posterior}
\Gamma_{\post}^N := (H^N_{\text{misfit}} + \Gamma_{\pr}^{-1})^{-1}.
\end{equation}

Hence,  we obtain an approximation of the $A$-optimality functional \eqref{eq:A_optimal} via projections
\begin{equation}\label{eq:A_optimal_projected}
\va{ &\min_{\iii}\ \varphi_N(\iii) = \Tr_{L^2(\Omega)} \left(\Gamma_{\post}^N(\iii)\right), \\
&\text{subject to } \iii \in \mathcal{D}.}
\end{equation}
It can be seen that the operator $H^N_{\text{misfit}}$ is nonnegative definite and is of trace class.  Hence,  following the same argument as in the proof of Proposition~\ref{prop:existence_minimizer},  for every $N \in \mathbb{N}$ there exists a minimizer $\iii_N^*$ to \eqref{eq:A_optimal_projected}.   Note that we do not need Assumption~\ref{ass:reg} for this purpose, since $H^N_{\text{misfit}}$ has finite dimensional range.

Our next question is whether the sequence of minimizers $\{\iii_N\}_{N \in \mathbb{N}}$ converges to a solution of \eqref{eq:A_optimal} in a certain sense.  To this end,  we first need the following lemma:

\begin{lemma}\label{lem:HS_convergence}
Let $A$ be a Hilbert-Schmidt operator. Assume that $\{P_k \}_{k \in \mathbb{N}}$ is the sequence of projections according to Lemma~\ref{lem:orthonormalbasis}, which converges to the identity pointwise, i.e. $ P_k x \to x$ for every $x \in \mathbb{N}$. Then $P_k A P_k \to A$ in the Hilbert-Schmidt norm.  Furthermore,  if $\{A_k\}_{k \in \mathbb{N}}$ is a sequence in $\HS(L^2(\Omega))$ that converges to $A$,  then $P_kA_kP_k \to A$ in HS norm.
\end{lemma}

\begin{proof}For every $m \in \mathbb{N}$,  since $A$ is a Hilbert-Schmidt operator and $P_m$ is a bounded operator,  the product $AP_m$ is also Hilbert-Schmidt.  If we consider the orthonormal basis $\{e_k\}_{k \in \mathbb{N}}$ given in Lemma \ref{lem:orthonormalbasis},  one has
\begin{equation}\label{eq:hs_norm}
\begin{aligned}
\norm{AP_m - A}_{\HS}^2  
&= \sum_{k=1}^{\infty} \norm{(AP_m - A)e_k}^2_{L^2(\Omega)} 
&= \sum_{k = m+1}^{\infty} \norm{Ae_k}^2_{L^2(\Omega)},
\end{aligned}
\end{equation}
which follows from the fact that $P_m e_k = e_k$ for all $k \le m$ and $P_m e_k = 0$ for all $k > m$.  
Since $A$ is Hilbert Schmidt,  the sum $\sum_{k = m+1}^{\infty} \norm{Ae_k}^2_{L^2(\Omega)}$ converges to zero, which implies $\norm{AP_m - A}_{\HS}^2 \to 0$ as $m \to \infty$.  
Replacing $A$ by $A^*$  we get $\norm{A^*P_m - A^*}_{\HS}^2 \to 0$. 
This implies
\begin{align*}
\norm{P_mA - A}_{\HS}^2 = \norm{(P_mA - A)^*}_{\HS}^2 
= \norm{A^*P_m - A^*}_{\HS}^2 \to 0,
\end{align*}
due to properties of the Hilbert-Schmidt norm and self-adjointness of the orthogonal projection $P_m$.  
The finite dimensional operator $P_mAP_m$ is a Hilbert-Schmidt operator and
\begin{equation}\label{eq:HS_convergence_1bis}
\begin{aligned}
\norm{P_mAP_m- A}_{\HS} 
&= \norm{P_m(AP_m - A) + (P_mA - A)}_{\HS}  \\
&\le \norm{P_m} \norm{AP_m - A}_{\HS} + \norm{P_mA - A}_{HS} \to 0
\end{aligned}
\end{equation}
as $n\to\infty$,  since the operator norm $\norm{P_m}$ is unity for every $m \in \mathbb{N}$. 

For a sequence $\{A_m\}_{m \in \mathbb{N}}$ in $\HS(L^2(\Omega))$ that converges to $A$ in the Hilbert-Schmidt norm, we have
\begin{equation}\label{eq:HS_convergence_2}
\begin{aligned}
\norm{P_mA_mP_m  - A}_{\HS} &\le  \norm{P_mA_mP_m  - P_mAP_m}_{\HS} + \norm{P_mAP_m - A}_{\HS} \\
&\le \norm{P_m}\norm{A_m - A}_{\HS}\norm{P_m} + \norm{P_mAP_m - A}_{\HS}\to 0,
\end{aligned}
\end{equation}
again by the fact that $\norm{P_m} = 1$ for all $m \in \mathbb{N}$.  
\end{proof}

The operator convergence result in Lemma~\ref{lem:HS_convergence} allows us to prove a stability result, following the technique in \cite{Duong_2023}.

\begin{theorem}\label{theo:stability}
Let Assumptions~\ref{assum:prior}, \ref{ass:reg} be satisfied.

For every $N \in \mathbb{N}$, the problem \eqref{eq:A_optimal_projected} admits a minimizer $\iii_N^*$.   

Furthermore,  the limit $\iii^*$ of any converging subsequence of minimizers $\{\iii_N^*\}_{N \in \mathbb{N}}$ of \eqref{eq:A_optimal_projected} is a minimizer of $\varphi$.
\end{theorem}

\begin{proof}
Existence of minimizers $\iii_N^*$ follows as in Proposition~\ref{prop:existence_minimizer}.

To prove the limiting result, we make use of $\Gamma-$convergence \cite{braides2002gamma}.  Consider a sequence $\{\iii_m\}_{m \in \mathbb{N}}$ that converges to $\iii$.  By the continuity of the map $\iii \mapsto \W_{\iii}^*\W_{\iii}$ (Proposition~\ref{prop:continuity}),  one has $\W_{\iii_m}^*\W_{\iii_m} \to \W_{\iii}^*\W_{\iii}$ in the HS norm.  Hence using Lemma~\ref{lem:HS_convergence} and the continuity of the trace map $g$ \eqref{eq:trace_continuity},  we obtain the convergence
\begin{align*}
\varphi_n(i_n) = g(P_n \W_{i_n}^*\W_{i_n} P_n) \to g(\W_{\iii}^*\W_{\iii}) = \varphi(\iii).
\end{align*} 
Next,  if we consider the constant sequence $\{\iii\}_{n \in \mathbb{N}}$,  it is also clear from Lemma~\ref{lem:HS_convergence} that
\begin{align*}
\limsup_{n \to \infty} \varphi_n(\iii) = \limsup_{n \to \infty} g(P_n \W_{\iii}^*\W_{\iii} P_n) = g(\W_{\iii}^*\W_{\iii}) = \varphi(\iii).
\end{align*}
Hence, the minimality of the limit of any convergent subsequence follows from the Fundamental Theorem of $\Gamma-$convergence \cite{braides2002gamma}. 
\end{proof}

\begin{remark}In Theorem~\ref{theo:stability}, the stability result is obtained without a rate of convergence. We note that with a suitable choice of subspaces $\{X_n\}_{n \in \mathbb{N}}$ for the projections, and under additional regularity assumptions, one would also obtains a convergence rate $\varphi_n(i_n) \to \varphi(i)$. For instance, we refer to \cite{walter_thesis_2019} for convergence rates obtained by projecting onto spaces spanned by the eigenvectors of the prior covariance operator $\Gamma_{\pr}$.
\end{remark}

\noindent {\bf Interpretation of the result.} The projection scheme introduced in Section \ref{sec:projection} can be considered as a posterior distribution obtained by combining variational regularization with regularization by projection.  More specifically,  assume that in \eqref{eq:abstract_inverse_problem},  we replace $\mathcal{W}_{\iii}$ by the restricted operator $\mathcal{W}_{\iii} P_N$ for some $N \in \mathbb{N}$,  namely
\begin{equation}\label{eq:inexact}
p_{\text{obs}} = \mathcal{W}_{\iii}P_N[a] + \eta.
\end{equation}

Since $\mathcal{W}_{\iii}$ is compact, the sequence $\mathcal{W}_{\iii} P_N$ converges to $\mathcal{W}_{\iii}$ in the operator norm \cite{conway_2000}.  Hence, our projection scheme can be interpreted as an inverse problem with an inexact forward operator or as regularization with discretization; relevant  results on such approaches for general inverse problems can be found in, e.g.  \cite{Neubauer_Scherzer_1990,  Aspri_Korolev_Scherzer_2020}, as well as the monographs \cite{Engl_Hanke_Neubauer_2000, Kirsch:2011}. 
In this case,  the posterior covariance operator is given by
\[ \left(\sigma^{-2} (\mathcal{W}_{\iii}P_N)^* \mathcal{W}_{\iii} P_N + \Gamma_{\pr}^{-1}\right)^{-1} = \left(\sigma^{-2} P_N \mathcal{W}_{\iii}^*\mathcal{W}_{\iii} P_N + \Gamma_{\pr}^{-1}\right)^{-1} = \Gamma_{\post}^N(\iii),\]
where the first equality holds since $P_N$ is an orthogonal projection.  This implies that $\Gamma^N_{\post}$ is the posterior covariance matrix for the inexact problem \eqref{eq:inexact}.

\section{Finite-dimensional discretization of OED}\label{sec:OED_discretized}
Building on the theoretical setup for the OED problem presented in Section~\ref{sec:OED}, we now examine its discretized version to illustrate the computational framework for the projection scheme.  Henceforth, discretized variables will be represented in boldface.  

\subsection{Discretization for the Bayesian inverse problem}
We assume that the space-discretized parameter dimension is $n$ and let $\{\phi_k\}_{k = 1,\ldots, n} \subset H^1(\Omega)$ be a Lagrangian FE basis corresponding to a discretization of $\Omega$.  The discretized absorption density is given by
\begin{align*}
a \approx \sum_{k=1}^n a_k \phi_k,
\end{align*}
with the coefficient vector $\bld{a} = (a_1,\ldots,  a_n)^{\top}$. 
We generally use the same letters for functions in $L^2(\Omega)$ as for the coefficient vector in their approximation, writing the latter in boldface for clarity.  
The $L^2-$inner product on $L^2(\Omega)$ 
is thus approximated by the weighted inner product on $\R^n$ with the mass matrix $\bo{M} \in \R^{n \times n}$, where $\bo{M}_{k,j} = (\phi_k,\phi_j)_{L^2(\Omega)}$,  i.e.  we have
\[ (f_1,f_2)_{L^2(\Omega)} \approx (\bo{f_1},\bo{f_2})_{\bo{M}}:= \bo{f_1}^{\top}\revision{M}{\bo{M}}\bo{f_2}
,\quad \forall f_1,  f_2 \in L^2(\Omega).\]
Consequently,  for a linear operator $\bo{A} : (\R^m,  (\cdot, \cdot)_{\bo{M}}) \to (\R^n,  (\cdot,\cdot)_{\bo{N}})$, the adjoint operator is given by $\bo{A}^* = \bo{M}^{-1} \bo{A}^{\top} \bo{N}$.  
We note that a (discretized) operator $\bo{A}$ is $\bo{M}$-symmetric,  i.e.  $(\bo{A}\bo{x},  \bo{y})_{\bo{M}} = (\bo{x},  \bo{A}\bo{y})_{\bo{M}}$ iff $\bo{MA}$ is symmetric (i.e.  $\bo{(MA)}^{\top} = \bo{MA}$).

\noindent {\bf Discretized forward operator.}  Using the traces of the ansatz functions $\phi_j$ on $\Sigma$ as a basis of $L^2(\Sigma)$, the matrix $\bo{B} \in \R^{n \times n}$ given by $\bo{B}_{i,j} = (\Tr_{\Sigma} \phi_i,  \Tr_{\Sigma}\phi_j)_{L^2(\Sigma)}$ defines an approximation of the $L^2(\Sigma)$ inner product
$$(g_1,g_2)_{L^2(\Sigma)} \approx (\bo{g_1},\bo{g_2})_{\bo{B}},\quad \forall g_1,  g_2 \in L^2(\Sigma).$$

Let $\bo{G}: (\R^n,  (\cdot,\cdot)_{\bo{M}}) \to (\R^{m_S},  (\cdot,\cdot)_{\bo{B}}) $ be the space-observation operator at a single observation time.  For $\bo{f} \in \R^n$ and $\bo{g} \in \R^{m_S}$,  one has
$$ (\bo{G}\bo{f},  \bo{g})_{\bo{B}} = (\bo{f},  \bo{G}^*\bo{g})_{\bo{M}},$$
which implies $\bo{G}^* = \bo{M}^{-1} \bo{G}^{\top} \revision{\bo{B}^{\top}}{\bo{B}}$.  Together with the time-observation discretization using quadrature weights of the composite Simpson's rule,  we obtain the discretization for the observation space $L^2(0,T;L^2(\Sigma))$,  and consequently,  the full discretization of the forward operator $\bo{W}_{\iii}$.  

\noindent {\bf Discretized prior.} For the trace-class prior,  we consider a discretization of the bi-Laplacian prior (see \cite{Stadler_2013}).  The matrix representation is given by $\bo{\Gamma}_{\pr} = \bo{A}^{-2}$,  where $\bo{A}^{-1} = \bo{K}^{-1}\bo{M}$ and $\bo{K}$ is the stiffness matrix given by
$$\bo{K}_{i,j} = \delta (\nabla \phi_i, \nabla \phi_j)_{L^2(\Omega)} + \gamma (\phi_i, \phi_j)_{L^2(\Omega)}.$$

In addition to trace-class smooth priors,  the discretized problem allows one to use a Gaussian prior with non-smooth covariance matrix which enhances the quality of edges in the reconstruction result.  Such priors can be defined via the Ornstein-Uhlenbeck covariance matrix \cite{Pulkkinen_2014}.
\begin{equation} \label{eq:ornstein}
(\bo{\Gamma}_{\pr})_{i,j} = \eta^2\exp\left(-\dfrac{|x_i - x_j|}{\ell}\right),
\end{equation}
where $x_i$ denotes the center of the $i-$th pixel, $\ell >0$ is the so-called correlation length and $\eta$ is the pixelwise standard deviation.

\noindent {\bf Discretized posterior.} The discretized posterior covariance is therefore given
\begin{align*}
\bo{\Gamma}_{\post} (\iii) = (\sigma^{-2}\bo{W}_{\iii}^*\bo{W}_{\iii} + \bo{\Gamma}^{-1}_{\pr})^{-1}
\end{align*}
and the A-optimality functional reads as
\begin{align*}
\varphi(\iii) = \tr_{\bo{M}}(\bo{\Gamma}_{\post} (\iii)) = \tr_{\bo{M}}\left[(\sigma^{-2}\bo{W}_{\iii}^*  \bo{W}_{\iii} + \bo{\Gamma}^{-1}_{\pr})^{-1}\right].
\end{align*}

\subsection{Evaluation of the discretized optimal design functional and its gradient} \label{sec:computation_functional}
In the following,  we focus on the computation of the discretized cost functional $\varphi(\iii) = \tr_{\bo{M}}[\bo{\Gamma}_{\post}(\iii)]$,  the approximation $\varphi_N(\iii) = \tr_{\bo{M}}[\bo{\Gamma}^N_{\post}(\iii)]$ for a fixed $N \in \mathbb{N}$ together with their gradients.  Let $\{e_k\}_{k = 1,\ldots,  n}$ be an orthonormal basis of $(\R^n,(\cdot,\cdot)_{\bo{M}})$.  Denote by $\bo{P}: (\R^n,  (\cdot,\cdot)_{\bo{M}}) \to \R^n$ the map which maps $\bld{a} \in (\R^n, (\cdot,\cdot)_{\bo{M}})$ to its Fourier coefficients,  i.e. $\bo{P}\bld{a} = [(\bld{a},  e_1)_{\bo{M}},\ldots,  (\bld{a} ,   e_n)_{\bo{M}}]$.  By a direct computation,  one can see that $\bo{P}^{-1}$ coincides with its adjoint $\bo{P}^*$ and $\bo{P}e_k = \bo{v}_k$,   where $\{\bo{v}_k\}_{k = 1,\ldots, n}$ is the standard basis of $\R^n$.  Hence, we have
\begin{equation}\label{eq:trace_comp_1}
\begin{aligned}
\tr_{\bo{M}}( \bo{\Gamma}_{\post}) 
&= \sum_{k=1}^{n} ( \bo{\Gamma}_{\post} e_k,  e_k)_{\bo{M}} \\
&= \sum_{k=1}^n ( \bo{\Gamma}_{\post} \bo{P}^{-1} \bo{v}_k,  \bo{P}^{-1} \bo{v}_k)_{\bo{M}} =\sum_{k=1}^n (\bo{P} \bo{\Gamma}_{\post} \bo{P}^{-1} \bo{v}_k,  \bo{v}_k)_2 \\
& = \tr_{\R^n}\left[\bo{P} \bo{\Gamma}_{\post} \bo{P}^{-1} \right] = \tr_{\R^n} \left[(\bo{P}\bo{H}_{\text{misfit}}\bo{P}^{-1} + \bo{P}\bo{\Gamma}_{\pr}^{-1}\bo{P}^{-1} )^{-1} \right].
\end{aligned}
\end{equation}
Similarly, when replacing $\bo{\Gamma}_{\post}$ by $\bo{\Gamma}^N_{\post}$ in \eqref{eq:trace_comp_1},  respectively,  we have
\begin{equation}\label{eq:trace_approximated}
\begin{aligned}
\tr_{\bo{M}}( \bo{\Gamma}^N_{\post})  = \tr_{\R^n} \left[(\bo{P}\bo{H}^N_{\text{misfit}}\bo{P}^{-1} + \bo{P}\bo{\Gamma}_{\pr}^{-1}\bo{P}^{-1} )^{-1} \right]
\end{aligned}
\end{equation}

 Now,  by selecting subspaces spanned by the eigenvectors of $\bo{\Gamma}_{\pr}$, noting that $\bo{\Gamma}_{\pr}$ is $(\cdot,\cdot)_{\bo{M}}$ symmetric, one can see that
\begin{equation}\label{eq:trace_comp_2}
\begin{aligned}
(\bo{P}\bo{\Gamma}_{\pr}^{-1}\bo{P}^{-1})_{j,k} = (\bo{P}\bo{\Gamma}_{\pr}^{-1}\bo{P}^{-1} \bo{v}_j,  \bo{v}_k)_2 = (\bo{\Gamma}_{\pr}^{-1}e_j,  e_k)_{\bo{M}} = \delta_{jk}\lambda_j^{-1},
\end{aligned}
\end{equation}
where $\{ \lambda_j \}_{j = 1,\ldots, n}$ are the corresponding eigenvalues of $\bo{\Gamma}_{\pr}$.  Similarly,
\begin{equation}
\begin{aligned}\label{eq:trace_comp_3}
(\bo{P}\bo{H}^N_{\text{misfit}}\bo{P}^{-1})_{j,k} &= (\bo{H}^N_{\text{misfit}}e_j,  e_k)_{\bo{M}} = \sigma^{-2} (\bo{W}_{\iii}\bo{P}_N e_j,  \bo{W}_{\iii}\bo{P}_N e_k)_{\bo{B}} \\ &= \va{ &0 &\text{ if } j > N \text{ or } k > N, \\  &\sigma^{-2} (\bo{W}_{\iii} e_j,  \bo{W}_{\iii} e_k)_{\bo{B}} &\text{ otherwise}.}
\end{aligned}
\end{equation}
We define $\widetilde{\bo{H}^N_{\text{misfit}}} \in \R^{N \times N}$ by $(\widetilde{\bo{H}^N_{\text{misfit}}})_{j,k}:= \sigma^{-2}(\bo{W}_{\iii} e_j,  \bo{W}_{\iii} e_k)_{\bo{B}}$.  From \eqref{eq:trace_comp_1}--\eqref{eq:trace_comp_3},  we have
\begin{align*}
\tr_{\bo{M}}( \bo{\Gamma}_{\post}^N)  = \tr_{\R^N}\left[(\widetilde{\bo{H}^N_{\text{misfit}}} + \diag(\lambda_j^{-1})_{j  = 1,\ldots,  N})^{-1}\right] + \sum_{k={N+1}}^n \lambda_k.
\end{align*}
Hence, one can see that computing the trace $\tr_{\bo{M}}( \bo{\Gamma}^N_{\post})$ can be reduced to computing the trace of an $N \times N$ matrix, which is low-dimensional. We note that a similar result was derived in \cite{walter_thesis_2019} in the context of optimizing sensor placement with pointwise measurements. Similarly, one can derive the computation of the cost functional for a more general sequence of projections. Details are provided in Appendix~\ref{append:projections}.

We next compute the gradient of $\varphi_N(\iii)$.  It can be seen that the action of the gradient $\varphi_N'(\iii)$ on some direction $\underline{\delta \iii}$ is given by
\[
\inner{\varphi_N'(\iii), \underline{\delta \iii}}_{H^1(0,T)} = -  \sigma^{-2}\tr_{\bo{M}}\left[\bo{\Gamma}_{\post}(\iii)^2(\bo{W}_{\iii}^* \bo{W}_{\underline{\delta \iii}} + \bo{W}_{\underline{\delta \iii}}^*  \bo{W}_{\iii})\right].
\]
However, an explicit representation of the gradient for direct computation appears to be unavailable.  Nevertheless, in practice, we are interested in continuous wave functions that are band-limited \cite{Lang_2019}.  Here,  we choose $\iii$ of the form 
\begin{equation}\label{eq:discr_i} 
\iii(t) = I_0\left[1 + \sum_{k=1}^{N_s} \coeffi_k \psi_k(t) )\right]
\end{equation}
where $\coeffi_k \in \R$ such that $\bo{\coeffi} = (\coeffi_1,\ldots,\coeffi_{N_s})$ belongs to an admissible set $\mathcal{A}$ and $\psi_k$ are functions that are band-limited.  Hence,  we have
\begin{align*}
\iii'(t) = \sum_{k=1}^{N_s} I_0\coeffi_k  \psi_k'(t),
\end{align*}
and consider the cost functional as a function of the finitely many coefficients $\coeffi_k$
\begin{align*}
\Psi_N(\bo{\coeffi}) = \varphi_N(\iii) = \tr[\bo{\Gamma}^N_{\post}(\iii)],\quad \bo{\coeffi} \in \mathcal{A}.
\end{align*}
Inserting the ansatz \eqref{eq:discr_i},
we can write $\bo{H}_{\text{misfit}}$ as
\begin{equation}\label{eq:Hess_comp}
\begin{aligned}
\sigma^{-2} \bo{W}_{\iii}^* \bo{W}_{\iii} 
& = \sigma^{-2} \left(\sum_{k=1}^{N_s} \coeffi_k \bo{W}_{\psi_k}\right)^* \left(\sum_{k=1}^{N_s} \coeffi_k \bo{W}_{\psi_k}\right)\\ 
&= \sigma^{-2} \sum_{k=1}^{N_s} \sum_{j=1}^{N_s} \coeffi_k \coeffi_j \bo{W}_{\psi_k}^*\bo{W}_{\psi_j},
\end{aligned}
\end{equation}
due to the linearity of $\bo{W}$ with respect to $\iii$.  In the case of projected misfit Hessian $\bo{H}_{\text{misfit}}^N$ \eqref{eq:proj_misfit},  the term $\bo{W}_{\psi_k}^*\bo{W}_{\psi_j}$ in \eqref{eq:Hess_comp} is replaced by $\bo{P}_N\bo{W}_{\psi_k}^*\bo{W}_{\psi_j}\bo{P}_N$.

Setting $\bo{A}_{kj} = \bo{W}_{\psi_k}^*\bo{W}_{\psi_j}$,  we have the following expression of the cost functional $\Psi = \Psi(\bo{\coeffi})$ together with its gradient:
\begin{align*}
\Psi(\bo{\coeffi}) &= \tr_{\bo{M}}\left[\bo{\Gamma}_{\post}(\bo{\coeffi})\right] = \tr_{\bo{M}} \left[ \left(\sigma^{-2} \sum_{k=1}^{N_s} \sum_{j=1}^{N_s} \coeffi_k \coeffi_j \bo{A}_{kj} + \bo{\Gamma}_{\pr}^{-1}\right)^{-1}\right]. \\ 
\dfrac{\partial \Psi}{\partial \coeffi_k} (\bo{\coeffi}) & = -\tr_{\bo{M}}\left[\bo{\Gamma}_{\post}(\bo{\coeffi})\left(\sigma^{-2} \sum_{j=1 }^{N_s} \coeffi_j (\bo{A}_{kj}+\bo{A}_{jk})  \right) \bo{\Gamma}_{\post}(\bo{\coeffi})\right]\\
& = -\sigma^{-2}\tr_{\bo{M}}\left[\bo{\Gamma}_{\post}(\bo{\coeffi})^2 \left(\sum_{j=1}^{N_s} \coeffi_j (\bo{A}_{kj}+\bo{A}_{jk})\right)  \right],\quad k  = 1,2,\ldots,  N_s,
\end{align*}
where in the last equality, we applied the cyclic property of the trace.

\input{numericalresults_revision_Truong}

\section{Conclusion} 

In this work, we have employed a fractional model for photoacoustic tomography to address the presence of power-law attenuation. We have considered a Bayesian approach to solve this problem, where an adjoint formulation is derived. Different choices of priors are studied in order to compare reconstruction quality. In addition, we have studied the influence of the laser intensity function on the reconstruction results, as well as chosen the optimized laser under certain conditions. The approach we have used employs an approximation via projection onto finite-dimensional subspaces.

One limitation of this method is that it depends on the choice of finite-dimensional subspaces for the projection. While the theoretical convergence result holds for arbitrary orthogonal projections, the practical performance depends on how well these subspaces align with the prior covariance operator. In particular, the accuracy of the reconstruction improves when the subspaces are chosen to reflect the eigenspaces spanned by the eigenvalues of the prior covariance operator.  Finally, the choice of the laser intensity function in this work is constrained to consist of a finite active frequency range due to the requirement that the signal be band-limited. In more general settings, where the design variable lies in an infinite-dimensional space, the optimization problem can be approached using methods such as conditional gradient algorithms.

\section*{Acknowledgements}
This research was funded in part by the Austrian Science Fund (FWF) [10.55776/DOC78].  
For open access purposes, the authors have applied a CC BY public  copyright license to any author accepted manuscript version arising from  this submission.
\revision{}{The authors wish to thank the reviewer for their careful reading of the manuscript and their valuable comments.}
\appendix
\section{Proof of Proposition \ref{prop:wellposed_forward}}\label{append:proof_wellposed}
\begin{proof}
For $f \in L^2(0,T;L^2(\Omega))$, existence and uniqueness of a solution follows from \cite[Theorem 7.1]{book_frac} with $m_*=N=0$, $\beta_{m_*}=\alpha_N=0$, $\beta_1=\alpha$.
To show that the result for a spatially less regular right hand side (as will be needed for the adjoint equation below), it suffices to use density as well as a basic energy estimate.
Multiplying \eqref{eq:wave_equation} with $\partial_t u$ and integration over space and time (actually this is the standard way of testing in case of a second order wave equation, cf., e.g., \cite{evans_1998}) 
using the identities 
\[
(\partial_t^2 u, \partial_t u)_{L^2(\Omega)} = \dfrac12\dfrac{d}{dt} \norm{\partial_t u}_{L^2(\Omega)}^2\,, \quad
\langle \nabla u, \partial_t \nabla u\rangle_{H^{-1}(\Omega),H_0^1(\Omega)} = \dfrac12\dfrac{d}{dt} \norm{\nabla u}_{L^2(\Omega)}^2\,,
\]
and $u(0)=0$, $\partial_t u(0)=0$, as well as the inequality
\[
\int_0^t (A^\gamma w)(\tau) w(\tau) \, d{\tau} \geq \cos \left( \frac{\pi\gamma}{2} \right) \| w \|_{H^{-\gamma/2}(0,t)}^2 \geq 0
\]
for $\gamma\in(0,1)$, cf. \cite{Eggermont:1988,VoegeliNedaiaslSauter:2016}, 
 we obtain that 
\[
\begin{aligned}
&\frac{1}{2}
\Bigl(\norm{c_0^{-1}\partial_t u(t)}_{L^2(\Omega)}^2
+\norm{\nabla u}_{L^2(\Omega)}^2\Bigr) 
+  \cos (\pi(1-\alpha)/2)\norm{\sqrt{\coeffi_0} \partial_t \nabla  u}_{H^{-(1-\alpha)/2}(0,t;L^2(\Omega))}^2\\
&\leq \int_0^t \langle f, \partial_t u \rangle_{H^{-1}(\Omega),H_0^1(\Omega)}\,d\tau,
\end{aligned}
\]
where the left hand side can be estimated from below by interpolation 
\[
\begin{aligned}
&\sup_{t\in(0,T)} \Bigl(\tfrac12|\tfrac{1}{c_0}\partial_t u(t)|_{L^2(\Omega))}^2
+  \cos ( \tfrac{\pi(1-\alpha)}{2} )\|\sqrt{b_0} \partial_t \nabla  u \|_{H^{-(1-\alpha)/2}(0,t;L^2(\Omega))}^2
\Bigr)\\
&\geq \min\left\{\frac{1}{T|c_0|_{L^\infty(\Omega)}^2},\cos ( \tfrac{\pi(1-\alpha)}{2} )|b_0|_{L^\infty(\Omega)}\right\}
\Bigl(\|\partial_t u\|_{L^2(0,T;L^2(\Omega))}^2+ \|\partial_t u \|_{H^{-(1-\alpha)/2}(0,T;H_0^1(\Omega))}^2\Bigr)\\
&\geq C(T,\alpha,\Omega) \|\partial_t u \|_{H^{-s(1-\alpha)/2}(0,T;H_0^s(\Omega))}^2 
\end{aligned}
\]
and the right hand side by 
\[
\begin{aligned}
&\sup_{t\in(0,T)} \int_0^t \langle f, \partial_t u \rangle_{H^{-s}(\Omega),H_0^s(\Omega)}\\
&\leq \frac{1}{2C(T,\alpha,\Omega)} \|f\|_{H^{s(1-\alpha)/2}(0,T;H^{-s}(\Omega))}^2
+\frac{C(T,\alpha,\Omega)}{2} \|\partial_t u\|_{H^{-s(1-\alpha)/2}(0,T;H_0^s(\Omega))}^2
\end{aligned}
\]
\end{proof}
\section{Singular values and Sobolev embedding} 
We recall some facts about singular values and Sobolev embeddings.  For details,  see \cite[Section II]{Gohberg_Krein_1988}.

Let $H, K$ be separable Hilbert spaces and $T \in B(H,K)$ a compact operator.  We denote by $s_k(T)$ the $k-$th singular value of $T$.  
\begin{lemma}\label{lem:singular}The following statements hold:
\begin{enumerate}
\item[(1)] If $B,  C$ are bounded operators,  then 
\begin{align*}
s_k(BTC) \le \norm{B}\norm{C} s_k(T),\quad \text{ for all } k \in \mathbb{N}.
\end{align*}
\item[(2)] $s_k(T) = s_k(T^*)$ for all $k \in \mathbb{N}$.
\end{enumerate}
\end{lemma}

We have the following result on the embedding between Sobolev spaces,  see for instance \cite[Lemma 3.3]{Nguyen_2011}.
\begin{lemma}\label{lem:embedding} Let $\gamma_1 > \gamma_2$ and $j: H^{\gamma_1}((0,T) \times \Sigma) \hookrightarrow H^{\gamma_2}((0,T) \times \Sigma)$ be the natural embedding.  There exists a constant $c > 0$ independent of $k$ such that
\begin{align*}
s_k(j) \le c j^{(\gamma_2 - \gamma_1)/d}.
\end{align*}
where $d$ is the space dimension of $\Omega$.
\end{lemma}

\section{Projection onto general subspaces}\label{append:projections}  In the following,  we show the computation of the cost functional for a general sequence of finite-dimensional subspaces in Section~\ref{sec:computation_functional}.  Recall that we are interested in computing the trace
\[
\tr_{\bo{M}}( \bo{\Gamma}^N_{\post})  = \tr_{\R^n} \left[(\bo{P}\bo{H}^N_{\text{misfit}}\bo{P}^{-1} + \bo{P}\bo{\Gamma}_{\pr}^{-1}\bo{P}^{-1} )^{-1} \right]
\]
where $\bo{P}: (\R^n,  (\cdot,\cdot)_{\bo{M}}) \to \R^n)$ is given by $\bo{P}\bld{a} = [(\bld{a},e_1)_{\bo{M}}, \ldots, (\bld{a},e_n)_{\bo{M}}]$.  Firstly,  since $\bo{P}\bo{\Gamma}_{\pr}^{-1}\bo{P}^{-1}$ is an $n \times n$ symmetric matrix,  can write
\begin{align}\label{eq:block_prior}
\bo{P}\bo{\Gamma}_{\pr}^{-1}\bo{P}^{-1} = \widetilde{\bo{\Gamma}_{\pr}} = \begin{bmatrix}
\bo{A} & \bo{B} \\ \bo{B}^{\top} & \bo{D}
\end{bmatrix},
\end{align}
where $\bo{A},  \bo{B},  \bo{D}$ are block matrices with $\bo{A} \in \R^{N \times N}$ and $\bo{D} \in \R^{ (n - N) \times (n - N)}$.  In particular,  $\bo{A}$ and $\bo{D}$ are symmetric and positive definite.

On the other hand,  the projected misfit Hessian is of the form
\begin{align}\label{eq:block_misfit}
\bo{P}\bo{H}^N_{\text{misfit}}\bo{P}^{-1} = \widetilde{\bo{H}^N_{\text{misfit}}}  = \begin{bmatrix}
\bo{H}^N_{\text{misfit}} & \bo{0} \\ \bo{0} & \bo{0}.
\end{bmatrix}
\end{align}

Combining \eqref{eq:block_prior} and \eqref{eq:block_misfit},  we have the following representation of $\bo{\Gamma}_{\post}$ which follows the well-known formula of inverses of block matrices (e.g.  \cite{Lu_2002}),
\begin{align*}
{{\begin{bmatrix}{\bo{H}^N_{\text{misfit}} + \bo{A}}&{\bo{B}}\\{\bo{B}^{\top}}&{\bo{D}}\end{bmatrix}}^{-1}={
\begin{bmatrix}\bo{L}^{-1}&-\bo{L}^{-1}\bo{B}\bo{D}^{-1}\\
-\bo{D}^{-1}\bo{B}^{\top}\bo{L}^{-1}& \bo{D}^{-1}+\bo{D}^{-1}\bo{B}^{\top}\bo{L}^{-1}\bo{B}\bo{D}^{-1}\end{bmatrix}},}
\end{align*}
where $\bo{L}:= {\bo{H}^N_{\text{misfit}} + \bo{A}}-\bo{B}\bo{D}^{-1}\bo{B}^{\top}$.  In particular,  we have
\begin{equation}\label{eq:trace_general}
\tr_{\bo{M}}{\bo{\Gamma}_{\post}} = \tr_{\R^N}(\bo{L}^{-1}) + \tr_{\R^{n-N}}(\bo{D}^{-1}+\bo{D}^{-1}\bo{B}^{\top}\bo{L}^{-1}\bo{B}\bo{D}^{-1}).
\end{equation}

We note that in \eqref{eq:trace_general},  $\bo{B}$,   $\bo{D}$,  $\bo{D}^{-1}$ can be precomputed offline,  i.e., before the optimization process.  Hence, only $\bo{H}^N_{\text{misfit}}$ needs to be updated in each iteration.  Since $\bo{H}^N_{\text{misfit}}$ is an $N\times N$ matrix (which is low dimensional),  its computation together with the inverse of $\bo{L} = {\bo{H}^N_{\text{misfit}} + \bo{A}}-\bo{B}\bo{D}^{-1}\bo{B}^{\top}$ is easy.

\bibliography{mybibfile,litBK}
\bibliographystyle{abbrv}
\end{document}

%% file: numericalresults_revision_Truong.tex
\section{Numerical results}\label{sec:numerical_results} Finally, we demonstrate some numerical examples to illustrate our theory. 

\subsection{Setting of the problem}\revision{}{To begin, we introduce the general setting that will be used throughout the examples. While we only consider in this work a mathematical example, the analysis still provides insight into practical photoacoustic tomography setups. In fact, the correct physical units for the setup will be provided in Section~\ref{sec:rescaling}.}

\revision{}{First, we define the experimental domain as $\Omega = [-1.5, 1.5] \times [-1.5, 1.5] \subset \mathbb{R}^2$, which serves as the setting for all numerical experiments.} The observation surface $\Sigma$ is defined as the boundary of the subdomain $\Omega_0 = \revision{}{[-0.6, 0.6] \times [-0.6, 0.6]}$. The observation time is $(0,  T)$ with $\revision{}{T = 0.1}$, \revision{}{and the time discretization is $\Delta t = 0.5 \cdot 10^{-3}$}. \revision{}{Our ground truth is the function $a(x) = 3 \cdot \bld{1}_{C}(x)$, where $C = [-0.5, 0.25] \times [-0.4, 0.15]$ and $\bld{1}_C$ is the characteristic function of $C$.}

The system governing the inverse problem is given in~\eqref{eq:wave_equation}, where $\revision{}{c = c_0 = 15}$, \revision{}{while $b$ and $\alpha$ will be given in} Section~\ref{sec:alpha}. \revision{}{By making the domain sufficiently large, we prevented perturbations due to spurious reflection on the Dirichlet boundary from affecting the pressure values on the observation surface $\Sigma$.} In order to solve the forward and adjoint problems,  we employ a Newmark time-stepping scheme for fractionally damped wave equations \cite{kaltenbacher_2022} in which the spatial domain is discretized by using the finite element method FEM with $\revision{}{n = 3721}$ spatial degrees of freedom.  

Our prior covariance matrices are chosen as follows:
\begin{equation}\label{eq:prior_choice}
\Gamma_{\pr} = (\gamma I - \delta\Delta)^{-2} \quad \revision{}{\text{ with }\gamma = 0.1 \text{ and }\delta = 10},
\end{equation}
where $I$ is the identity operator and $\Delta$ is the Laplacian operator on the given space, which yields a smooth prior in the infinite-dimensional setting. In the discretized setting, we use the Ornstein–Uhlenbeck prior as defined in \eqref{eq:ornstein}, with its covariance matrix parameterized  $\eta = 0.1$ and $\ell = 0.1$.  These parameters are heuristically chosen by empirical observations and prior knowledge of the given model following the literature \cite{kaltenbacher_2022,Pulkkinen_2014}. The implementation is done in Python with hIPPYlib library \cite{Villa_2021}. The system is solved using a conjugate gradient method implemented in the same library. \\
The code can be found at \url{https://github.com/hphuoctruong/fracPAT}.

\subsection{Influence of $\alpha$ and the choice of priors on reconstruction results}\label{sec:alpha}\text{}
In this first example, we study the influence of $\alpha$ on the reconstruction results. In the forward model \eqref{eq:wave_damped} as well as the adjoint model \eqref{eq:wave_equation_adjoint}, we consider $\alpha \in \{ 0.3,  0.8 \}$, \revision{}{corresponding to the weak and strong damping cases, respectively}. The damping coefficient in each case is $b = (-2c_0 r_0)/\cos (\pi (\alpha + 1)/2)$ \cite{Baker_2022},  where we choose $\revision{}{r_0 = 10^{-3}}$.
The intensity function is chosen of the form
\begin{equation}\label{eq:ref_intensity}
\iii_0 = \iii_0(t) = I_0\left[1 + \sin(\omega t)\right]
\end{equation}
 with $\omega = 100\pi$.
We consider Gaussian noise drawn from $\mathcal{N}(0,\sigma^2 I)$ with $\sigma^2 = 10^{-2}$.  

\noindent {\bf Influence of the choice of the differentiation order $\alpha$ in the damping.}  \revision{}{In order to illustrate the effect of $\alpha$ on the wave propagation, we demonstrate the state variables, together with the ground truth, in both cases $\alpha \in \{ 0.3, 0.8 \}$ for $I_0 = 0.5 \cdot 10^2$. The plots of the snapshot of the state variables at $t \in \{0.02, 0.04, 0.06, 0.08, 0.1\}$ are shown in Figure~\ref{fig:state_03} and Figure~\ref{fig:state_08}, respectively. 
With $\alpha=0.3$, we observe slower decay of the pressure, indicating weaker damping and more persistent energy propagation. In contrast, when $\alpha = 0.8$, the system experiences rapid smoothing and attenuation of wavefronts, reflecting stronger damping and faster energy loss.}

\begin{figure}[h]
	\centering
	\includegraphics[width=\textwidth]{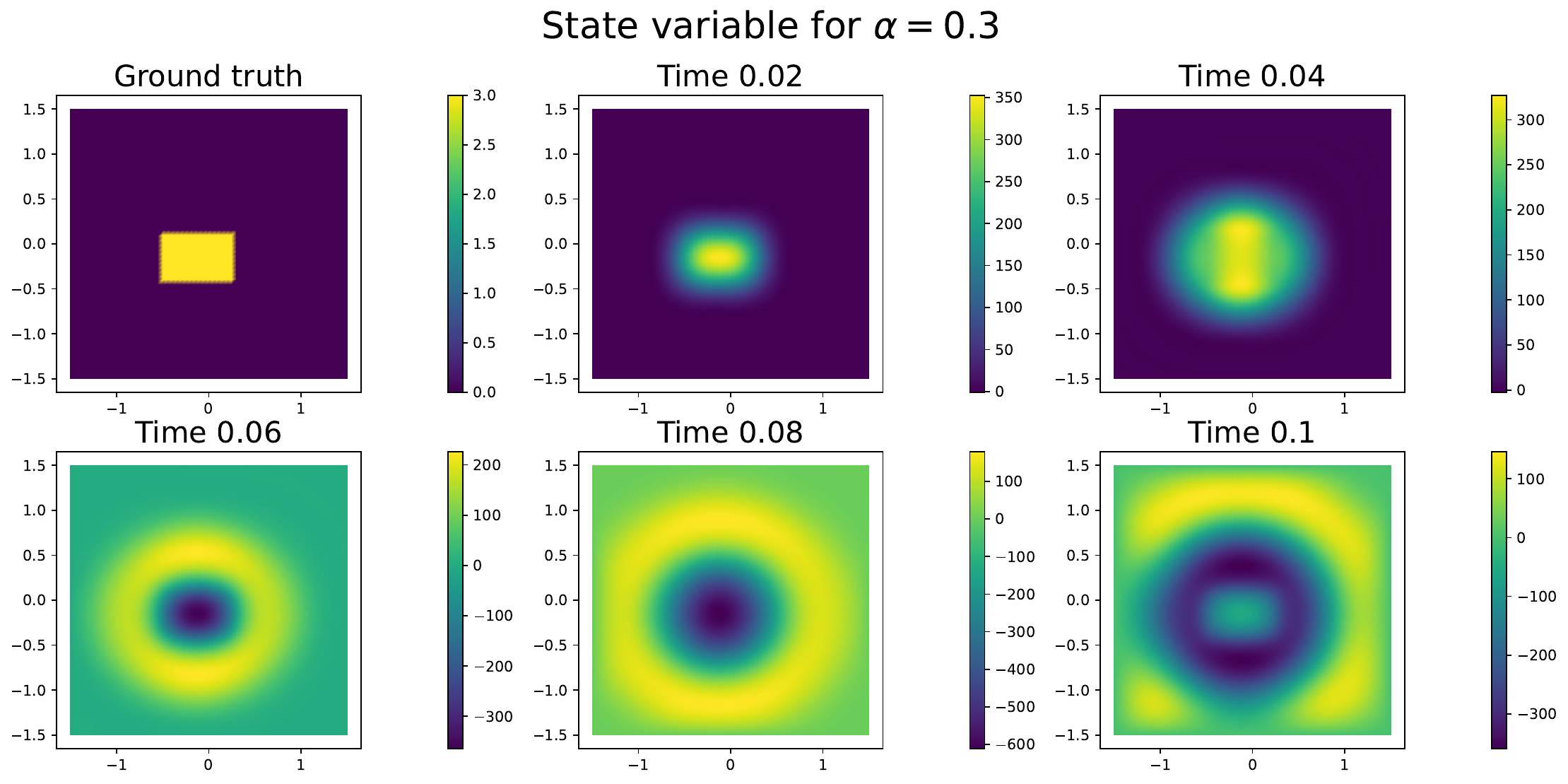}
	\caption{Snapshots of the state variable for $\alpha = 0.3$}
	\label{fig:state_03}
\end{figure}

\begin{figure}[h]
	\centering
	\includegraphics[width=\textwidth]{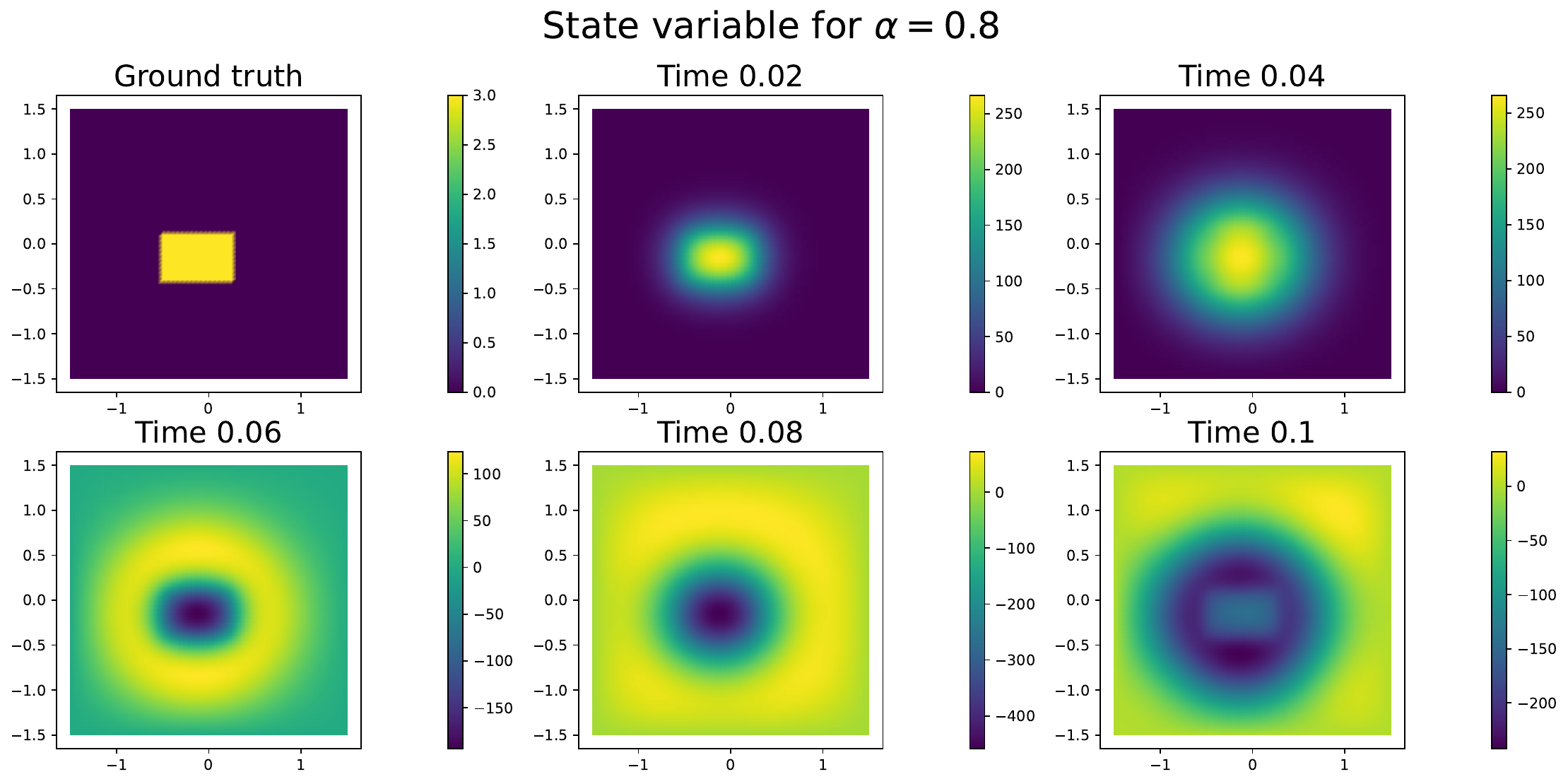}
	\caption{Snapshots of the state variable for $\alpha = 0.8$}
	\label{fig:state_08}
\end{figure}

\revision{}{We compare the MAP estimates given by \eqref{eq:map_estimator} using the bi-Laplacian prior in \eqref{eq:prior_choice}}.  Reconstruction results for $\alpha \in \{0.3,  0.8\}$, together with the true solution, are given in Figure \ref{fig:lowenergy}.  Here, one observes that with the weak damping $\alpha = 0.3$,  the ground truth is fairly well-reconstructed.  The result is smooth due to the smoothing property of the bi-Laplacian prior \eqref{eq:prior_choice}.  In the case of strong damping with $\alpha = 0.8$, the reconstruction quality is noticeably poorer, due to the significant loss of information by attenuation. \revision{We numerically evaluate the reconstruction results for a realization of the Gaussian noise using the relative $L^2(\Omega)$ error $\norm{a - a_0}_{L^2(\Omega)}/\norm{a_0}_{L^2(\Omega)}$,  which is given in Table~\ref{tab:bilaplacian}.}{To evaluate the reconstruction results, for the sake of simplicity we use the relative $L^2(\Omega)$-error $\norm{a - a_0}_{L^2(\Omega)}/\norm{a_0}_{L^2(\Omega)}$, where the data are perturbed by a single realization of Gaussian noise. }
	

In order to improve the reconstruction quality \revision{in the stronger damping case}{in both cases}, we increase the energy $I_0$ of the excitation function by a factor of four to $\revision{I_0 = 4\cdot 10^2}{I_0 = 2\cdot 10^2}$.  Reconstruction results in the setting with higher energy are shown in Figure~\ref{fig:highenergy}.  Here, one can see that there is an obvious improvement over the reconstructions in Figure~\ref{fig:lowenergy}.   The reconstruction errors in both cases are given in Table~\ref{tab:bilaplacian}. In addition, since only the temporal change $\iii'$,  rather than the values $\iii$ enter the model, it is obvious that improvements can be made by increasing the frequency of the excitation function.
We note that safety guidelines regarding the maximum light fluence or power limitations of the excitation laser must be closely monitored in practical situations, though, which is why we put a constraint on the magnitude of  $\iii'$.

\begin{figure}[h]
    \centering
        \includegraphics[width=0.32\textwidth]{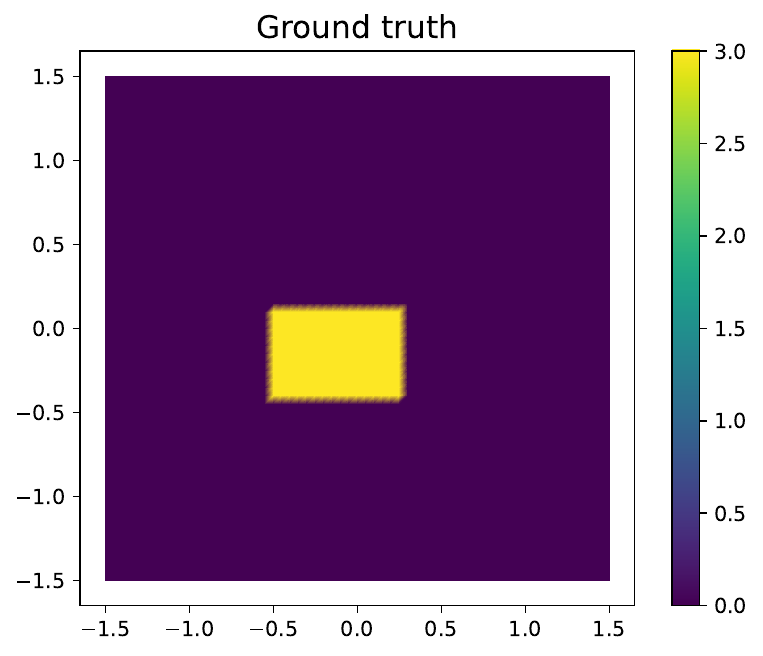}
        \includegraphics[width=0.32\textwidth]{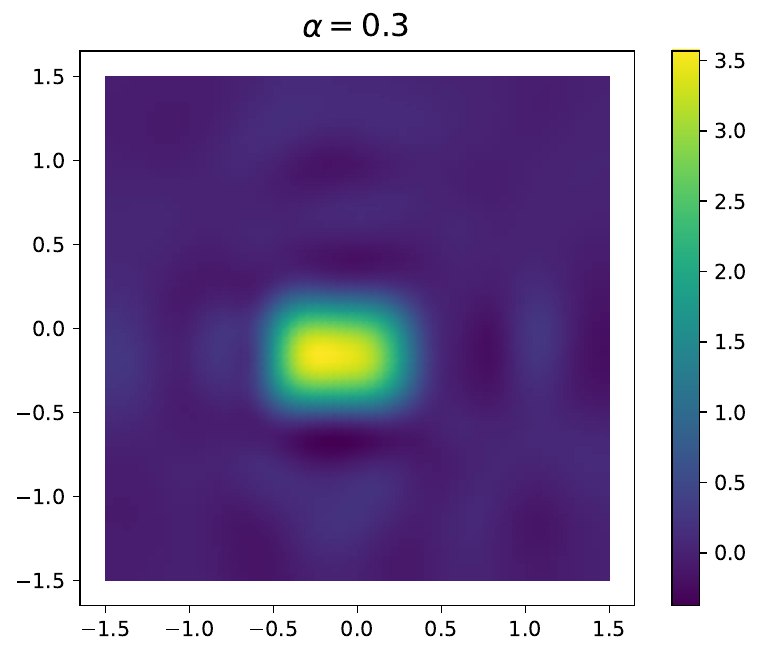}
        \includegraphics[width=0.32\textwidth]{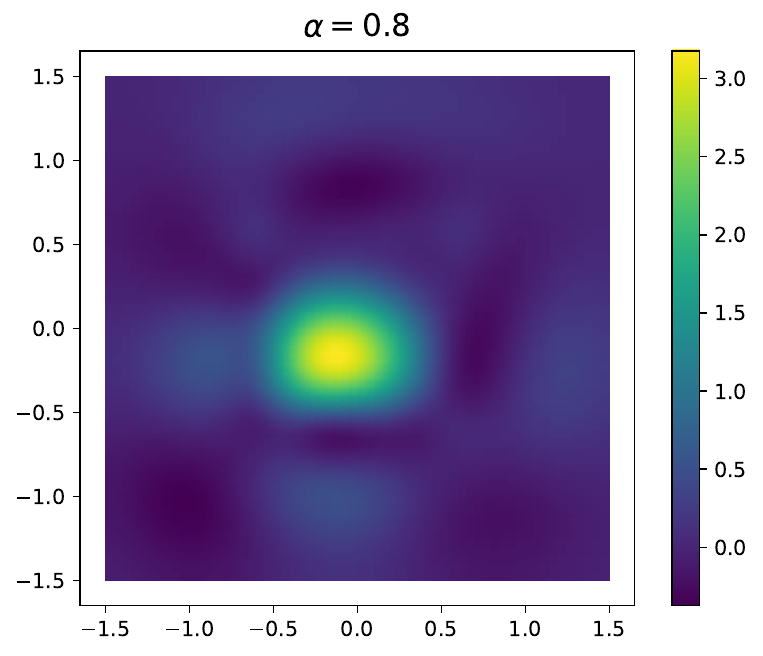}
    \caption{Ground truth (left) and reconstruction results for $\alpha = 0.3$ (middle) and $\alpha = 0.8$ (right) with $I_0 = 0.5 \cdot 10^2$, \revision{}{obtained using the bi-Laplacian prior}}
    \label{fig:lowenergy}
\end{figure}

\begin{figure}[h]
    \centering
        \includegraphics[width=0.32\textwidth]{u_init.pdf}
        \includegraphics[width=0.32\textwidth]{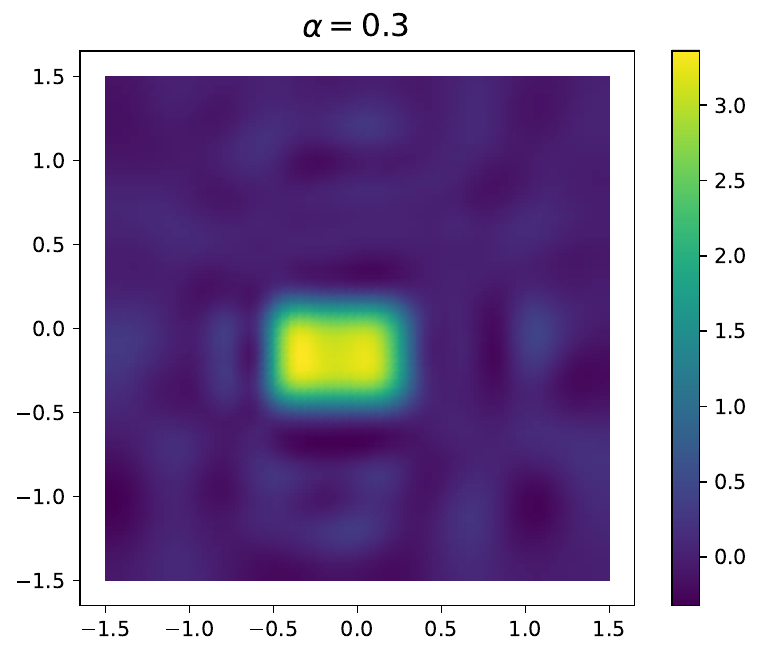}
        \includegraphics[width=0.32\textwidth]{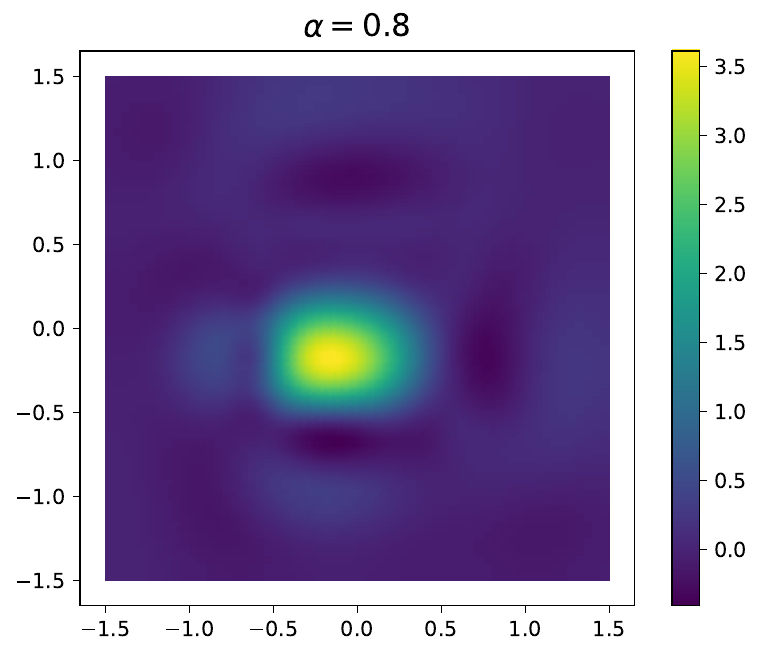}
    \caption{Reconstruction results for $\alpha = 0.3$ (middle) and $\alpha = 0.8$ (right) with $I_0 = 2\cdot 10^2$, \revision{}{obtained using the bi-Laplacian prior}}
    \label{fig:highenergy}
\end{figure}

\begin{table}[h!]
\centering
\begin{tabular}{l|cc}
\multicolumn{1}{c|}{} & \begin{tabular}[c]{@{}c@{}}Weak damping\\ $\alpha = 0.3$\end{tabular} & \begin{tabular}[c]{@{}c@{}}Strong damping\\ $\alpha = 0.8$\end{tabular} \\ \hline
$I_ 0 = 0.5\cdot 10^2$				& $\revision{}{35.76 \%}$ 	& $\revision{}{48.51 \%} $                                                                    \\
$I_0 = 2 \cdot 10^2$   & $\revision{}{31.53 \%}$   & $\revision{}{41.90 \%}$                                                                   
\end{tabular}
\caption{Reconstruction errors with bi-Laplacian prior for different strengths of damping and different excitation amplitudes.}
\label{tab:bilaplacian}
\end{table}

\noindent {\bf Influence of the choice of priors.} Next, we investigate the selection of the prior distribution, acknowledging its substantial influence on the MAP estimates.  In this context, we compare the reconstruction results using the bi-Laplacian prior in the first example with those obtained using a Gaussian prior with an Ornstein-Uhlenbeck covariance matrix.  We study both strong and weak damping cases corresponding to $\alpha \in \{ 0.3,  0.8 \}$.  We choose the intensity function as in \eqref{eq:ref_intensity} with $I_0 = 0.5 \cdot 10^2$.  With the same realization of noise as in the first example, a comparison of the reconstruction errors is given in Table~\ref{tab:bilaplacianvsornstein}.
\begin{table}[h!]
\centering
\begin{tabular}{l|cc}
\multicolumn{1}{c|}{}	 & Bi-Laplacian prior 											& Ornstein-Uhlenbeck prior \\ \hline
$\alpha = 0.3$				& $\revision{}{35.76 \%}$ 				  & $\revision{}{26.66 \%}$                \\
$\alpha = 0.8$				& $\revision{}{48.51 \%}$          	 & $\revision{}{40.09 \%}$              
\end{tabular}
\caption{Reconstruction error for different strengths of damping and different priors with $I_0 = 0.5\cdot10^2$.}
\label{tab:bilaplacianvsornstein}
\end{table}
We observe that the MAP estimate using the Ornstein-Uhlenbeck prior performs notably better than the one using the bi-Laplacian prior, particularly in preserving the sharp edges of the ground truth. In contrast, the bi-Laplacian prior produces a smoother result, reflecting its natural tendency to smooth the reconstruction.  Nevertheless, in the strong damping case, the reconstruction quality in both scenarios is significantly poorer, primarily due to the substantial loss of energy as the wave travels through the medium.  For illustration,  we provide the MAP estimates in Figure~\ref{fig:ornsteinprior} with $I_0 = 0.5 \cdot 10^2$ and $\alpha = 0.3$.
\begin{figure}[ht]
    \centering
        \includegraphics[width=0.32\textwidth]{u_init.pdf}
        \includegraphics[width=0.32\textwidth]{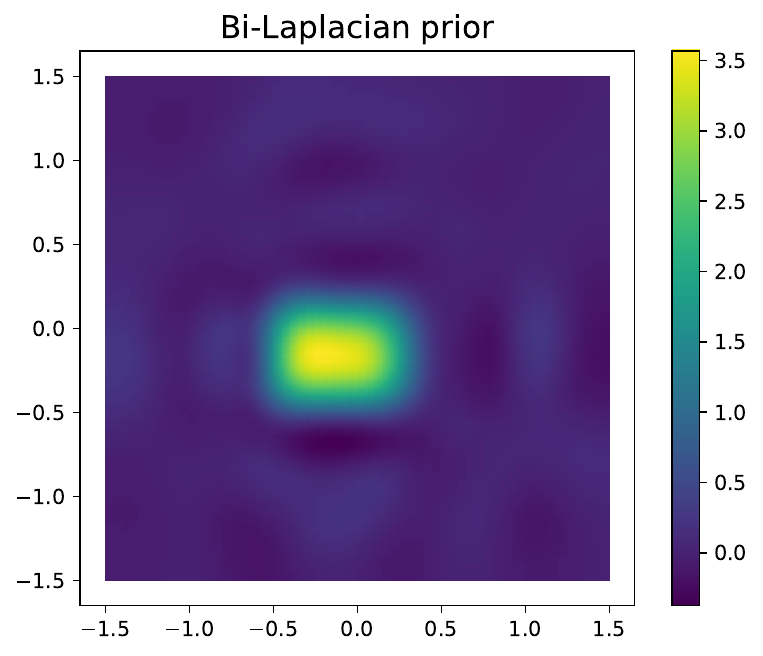}
        \includegraphics[width=0.32\textwidth]{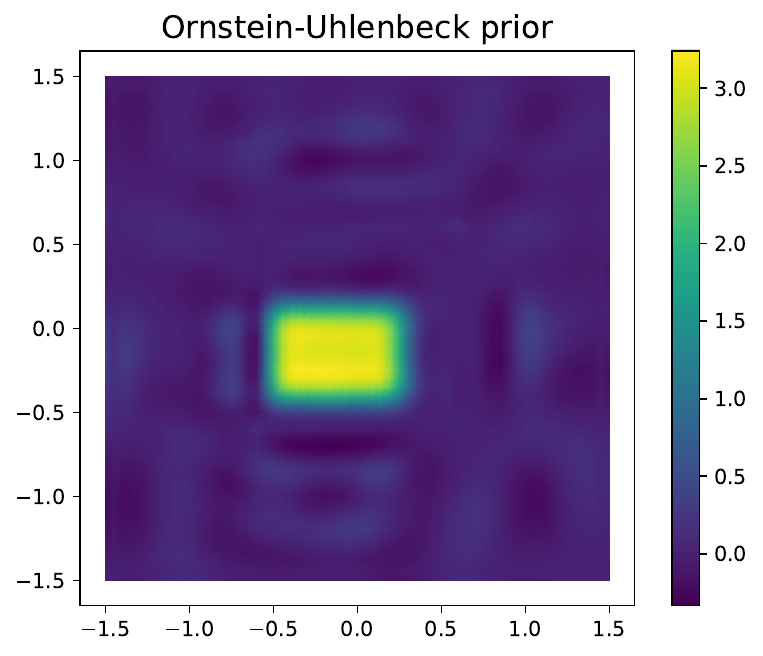}
    \caption{Reconstruction results with bi-Laplacian prior (middle) and Ornstein-Uhlenbeck prior (right) with $I_0 = 0.5 \cdot 10^2$ and $\alpha = 0.3$ }
    \label{fig:ornsteinprior}
\end{figure}

\subsection{Optimization of laser intensity function}
We next study the problem of optimizing the laser intensity function which is considered as a design variable. Following Section~\ref{sec:computation_functional},  the intensity function is assumed to be band-limited.  To obtain this,  we consider a signal that has a finite number of $K$ active frequencies,  namely
\begin{equation}\label{eq:i_form}
\iii = \iii(t;I,  \coeffi_1,\ldots,\coeffi_K) = I\left[1 +
\sum_{k=1}^K \coeffi_k\sin(\revision{}{\omega_k} t)\right],
\end{equation}
where $\omega_k,  k = 1,\ldots, K$ are fixed and $I,  \bo{\coeffi}= (\coeffi_1,  \ldots,  \coeffi_K)$ are variables. 
The optimal laser intensity that we are looking for should in addition fulfill the following conditions:
\begin{itemize}
\item[$\bullet$] $\iii \ge 0$,  in view of the fact that only nonnegative laser intensities make sense physically.  
\item[$\bullet$] $\iii$ has bounded $H_1$ norm to satisfy bounds on the energy. Such a constraint of the $H^1-$norm and thus on the high frequency content of the signal also make sense in view of the fact that higher frequencies undergo more significant attenuation while propagating through the medium.
\end{itemize}
Hence,  we consider the set
\begin{equation}\label{eq:admissible_set}
\mathcal{D}:= \left\{ \iii \in H^1(0,T): \iii \text{ is of the form }\eqref{eq:i_form},  |\bo{\coeffi}|_1 \le 1,  \norm{\iii}_{H^1(0,1)} \le M := 4\norm{\iii_0} \right\},
\end{equation}
where the constraint $|\bo{\coeffi}|_1 \le 1$ guarantees the positivity of $\iii$ and the latter ensures that the norm of 
$\iii$ is sufficiently bounded yet allows room for variations in the optimization process.  Note that $\mathcal{D}$ is a (weakly) compact subset of $H^1(0,T)$.

Turning to the optimality criterion,  we consider the approximated $A$-optimality cost functional $\varphi_N$ via projections given in ~\eqref{eq:A_optimal_projected}.  We choose the projections onto finite-dimensional subspaces spanned by the eigenvectors of the prior covariance operators of both bi-Laplacian prior and Ornstein-Uhlenbeck prior.  The decays of the eigenvalues in both cases are illustrated in Figure~\ref{fig:decay}.

\begin{figure}[h]
\includegraphics[width=1.\textwidth]{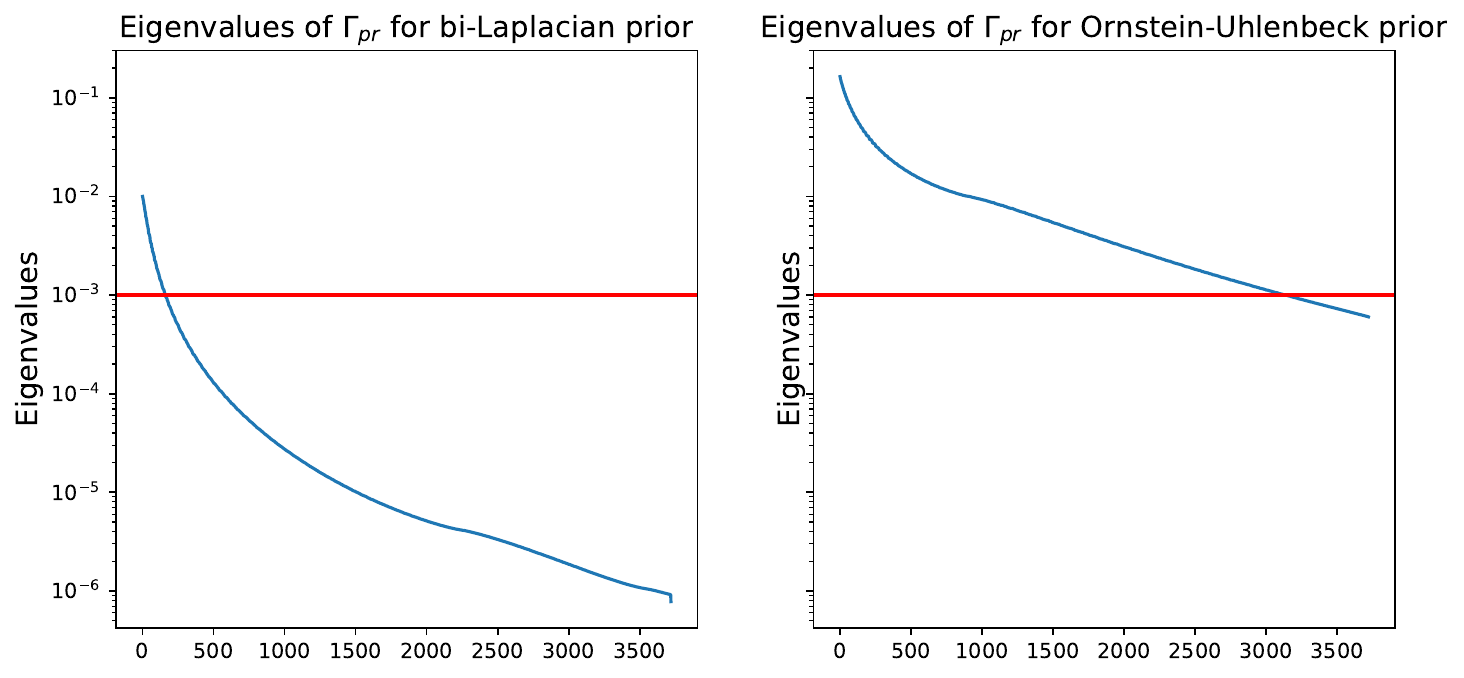}
\caption{Decay of eigenvalues of $\bo{\Gamma}_{\pr}$ for bi-Laplacian prior (left) and Ornstein-Uhlenbeck-prior (right)}
\label{fig:decay}
\end{figure}
In the first case,  observing that the decay of eigenvalues  is sufficiently fast,  we consider the projection onto $X_N$, spanned by the first $N$ eigenfunctions, with $N = 160$.  Since $\iii$ in this setting is parametrized by $I, \bo{\coeffi}$ as in \eqref{eq:admissible_set},  we solve the optimization problem
\begin{equation}
\va{ &\Psi_N(I,  \bo{\coeffi}) = \varphi_N(\iii) = \tr_{\bo{M}} (\bo{\Gamma}^N_{\post}(\iii)) \to \min, \\
&\text{ subject to } \iii \in \mathcal{D}.}
\end{equation}
In the optimization process,  we use the initial guess $(I_0,  \bo{\coeffi}_0)$ with $\revision{}{I_0 = 0.5 \cdot 10^{2}}$ and $\bo{\coeffi}_0 = (1,0,0,\ldots)$,  which corresponds to the reference intensity function $\iii_0$ used in Section~\ref{sec:alpha}. 
The number of frequencies was chosen as $K=5$; a larger choice did not yield any significant changes. \revision{}{We consider the set of frequencies $\omega \in \{ 20 \pi, 30 \pi, 40\pi, 50 \pi, 60 \pi  \}$, with their physical interpretation provided in Section~\ref{sec:rescaling}.}

We note that $\Psi_N$ is monotone in $I$ and therefore $\Psi_N(I,  \bo{\coeffi}) \ge \Psi(4I_0,  \bo{\coeffi})$ due to the constraint $\norm{\iii}_{H^1(0,T)} \le 4 \norm{\iii_0}_{H^1(0,1)}$,  which reduces the problem to optimizing over $\bo{\coeffi}$. 

\revision{}{In the weak damping case $\alpha = 0.3$}, the optimization process yields the coefficients $(I_{\text{opt}},  \bo{\coeffi}_{\text{opt}})$ with $\revision{}{I_{\text{opt}} = 2 \cdot 10^2}$ and $\bo{\coeffi}_{\text{opt}} = \revision{}{(0.1329, 0.2173, 0.2291, 0.2184, 0.2023)}$, from which the optimal intensity function $\iii_{\text{opt}}$ results according to \eqref{eq:i_form}. \revision{}{Intuitively, this represents a trade-off between different aspects of wave propagation in attenuated media: high frequencies are known to provide sharp but attenuated signals, while low frequencies offer more penetration but less resolution.} Values of the $A$-optimality cost functional for $\iii_0$ and $\iii_{\text{opt}}$ are shown in Table~\ref{tab:optimal_values}, where it is evident that the obtained solutions significantly improve the approximated average reconstruction results. 

 We note that due to the nonconvex nature of the cost functional,  the optimal solution $\iii_{\text{opt}}$ is only a local solution.  In the case of the Ornstein-Uhlenbeck prior, although the decay is not sufficiently rapid, we still project onto the space spanned by the first $N=160$ eigenfunctions to render the minimization problem computationally tractable.  We again obtain an optimal solution $(I_{\text{opt}}, \bo{\coeffi}_{\text{opt}})$ with $\revision{}{I_{\text{opt}} = 2 \cdot 10^2}$ and $\revision{}{\bo{\coeffi}_{\text{opt}} = (0.1269, 0.2216, 0.2338,  0.2195,  0.1983)}$. Results are shown in Table~\ref{tab:optimal_values},  where one can also see that this also significantly improves on $A$-optimality.  

\revision{}{To demonstrate the effect of the initial and optimal intensity functions on wave propagation, we plot snapshots of the wave field resulting from the reconstructed absorption density function $a$ at different time steps. Here, we again select the time instances $t \in \{ 0.02, 0.04, 0.06, 0.08, 0.1\}$. Results are shown in Figure~\ref{fig:state_i0} and Figure~\ref{fig:state_iopt}.}

\begin{figure}[h]
	\centering
	\includegraphics[width=\textwidth]{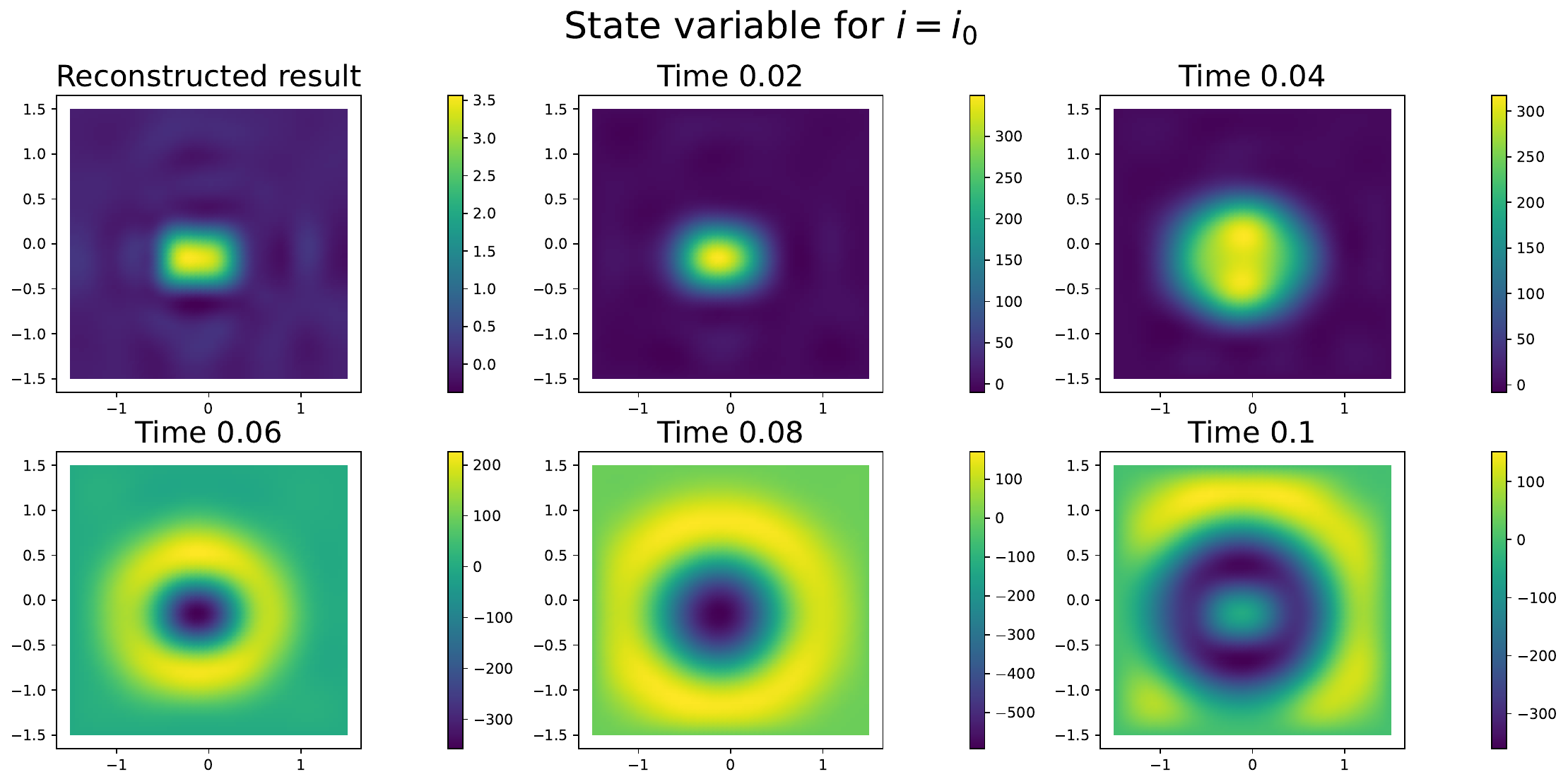}
	\caption{Snapshots of the state variable at $t \in \{ 0.02, 0.04, 0.06, 0.08, 0.1\}$ for $a$ reconstructed with $i = i_0$ in the weak damping case}
	\label{fig:state_i0}
\end{figure}

\begin{figure}[h]
	\centering
	\includegraphics[width=\textwidth]{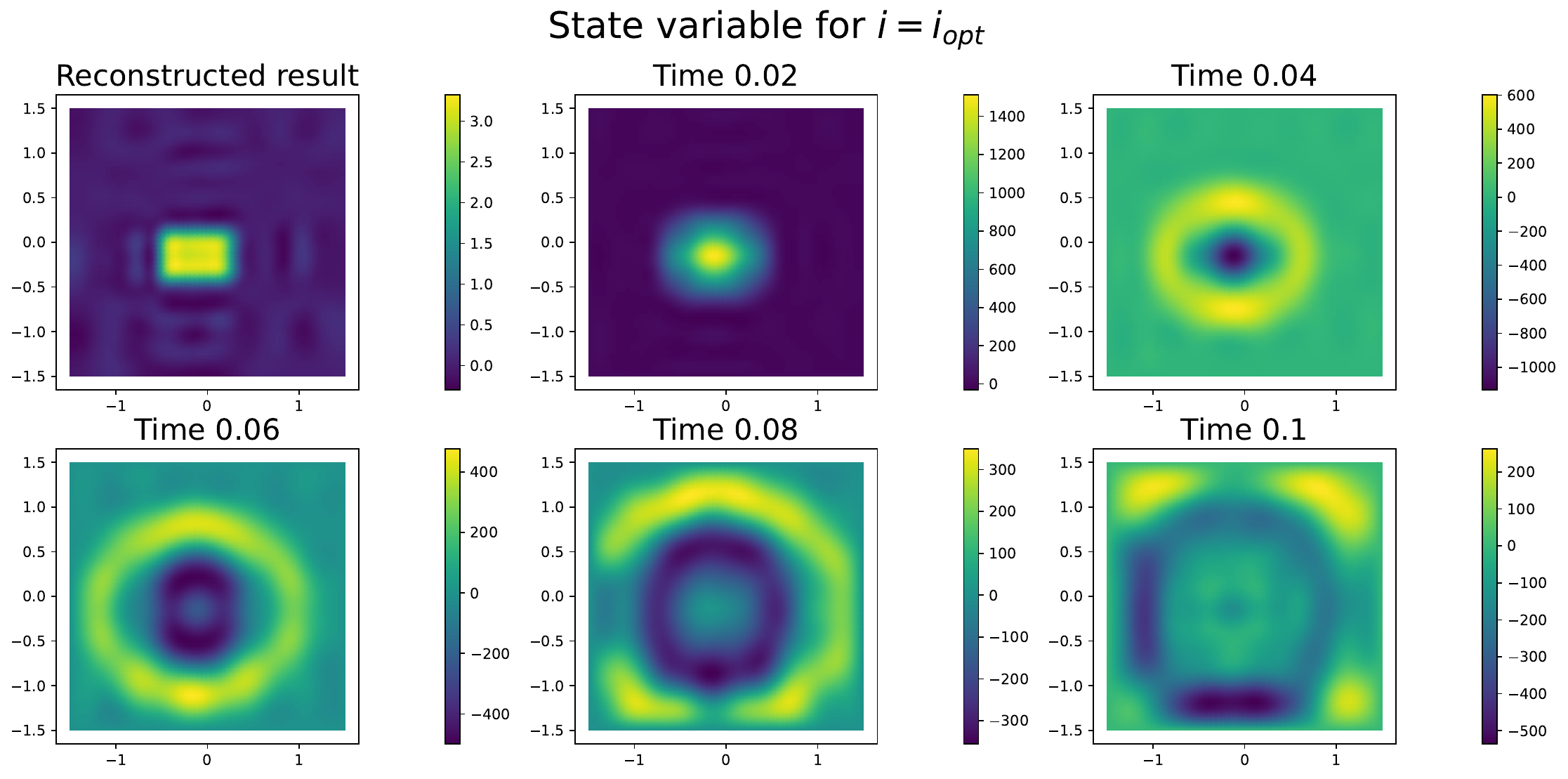}
	\caption{Snapshots of the state variable at $t \in \{ 0.02, 0.04, 0.06, 0.08, 0.1\}$ for $a$ reconstructed with $i = i_{\text{opt}}$ in the weak damping case}
	\label{fig:state_iopt}
\end{figure}

Finally, to numerically verify the improvement, we consider the reconstruction problem with the realization of noise considered in Section~\ref{sec:alpha} and compare reconstruction results using $\iii_0$ and $\iii_{\text{opt}}$. The reconstruction results, summarized in Table~\ref{tab:optimal_values}, demonstrate a significant enhancement in accuracy. Specifically, in the bi-Laplacian case, the reconstruction error is reduced by $\revision{}{9.45\%}$, indicating a notable performance gain. \revision{}{The Ornstein–Uhlenbeck case shows a more significant improvement, with the reconstruction error decreasing by $\revision{}{10.28 \%}$.  Both show the relative improvement of $\revision{}{26.42\%}$ and $\revision{}{38.56} \%$, respectively. }

Notably, the maximum frequency (under the given constraints) excitation function $\iii$ determined by the coefficients $\revision{}{(0,0,0,0,1/3)}$ is not optimal.  Reconstruction errors with this choice are $\revision{}{29.71\%}$ and $\revision{}{19.05\%}$, respectively. \revision{}{In addition, one clearly observes that only increasing the energy as in Table~\ref{tab:bilaplacian} does not yield the best reconstruction results}. This shows that choosing a more sophisticated intensity function by means of mathematical optimization indeed pays off.
\begin{table}[h]
\begin{tabular}{c|cc|cc}
                                                                                & \multicolumn{2}{c|}{Bi-Laplacian prior}                            & \multicolumn{2}{c}{Ornstein-Uhlenbeck prior}                       \\ \cline{2-5} 
\multicolumn{1}{l|}{}                                                           & \multicolumn{1}{c|}{Relative error} & $\varphi_N(\iii)$ & \multicolumn{1}{c|}{Relative error} & $\varphi_N(\iii)$ \\ \hline
\begin{tabular}[c]{@{}c@{}}Initial guess\\ $I_0 = 0.5 \cdot 10^2,  \bo{\coeffi}_0$\end{tabular}            
& \multicolumn{1}{c|}{$\revision{}{35.76 \%}$}              & $\revision{}{0.192}$              
& \multicolumn{1}{c|}{$\revision{}{26.66 \%}$}              & $\revision{}{2.198}$              \\ \hline
\begin{tabular}[c]{@{}c@{}}Optimal solution\\ $I_{\text{opt}} = 2\cdot 10^2,  \bo{\coeffi}_{\text{opt}}$\end{tabular} 
& \multicolumn{1}{c|}{$\revision{}{\bld{26.31 \%}}$}             & $\revision{}{\bld{0.042}}$              
& \multicolumn{1}{c|}{$\revision{}{\bld{16.38 \%}}$}              & $\revision{}{\bld{0.140}}$             
\end{tabular}
\caption{Reconstruction errors with initial and optimal intensity functions for $\alpha = 0.3$}
\label{tab:optimal_values}
\end{table}

\revision{}{
In the strong damping case $\alpha = 0.8$, the optimization process corresponding to the bi-Laplacian prior yields the coefficients $(I_{\text{opt}},  \bo{\coeffi}_{\text{opt}})$ with $\revision{}{I_{\text{opt}} = 2 \cdot 10^2}$ and $\bo{\coeffi}_{\text{opt}} = \revision{}{(0.3306, 0.2598, 0.0000, -0.1529, -0.2567)}$, from which the optimal intensity function $\iii_{\text{opt}}$ results according to \eqref{eq:i_form}. This observation suggests an different combination of low and high frequencies in order to yield good reconstruction results, in comparision to the one obtained in the weak damping case. The reconstruction resulting from the excitation function $i$ with coefficients $(0, 0, 0, 0, 1/3)$ yields poorer performance compared to the one based on $(1, 0, 0, 0, 0)$, with reconstruction errors $47.83\%$ and $41.90\%$, respectively. This behavior aligns with physical expectations, where the system's response is dominated by modes that are less sensitive to damping. The balance between low and high frequencies of the optimal intensity function for the Ornstein-Uhlenbeck prior yields the solution $(I_{\text{opt}},  \bo{\coeffi}_{\text{opt}})$ with $\revision{}{I_{\text{opt}} = 2 \cdot 10^2}$ and $\bo{\coeffi}_{\text{opt}} = \revision{}{(0.1220, 0.2072, 0.2721, 0.2157, 0.1836)}$. Results are illustrated in Table~\ref{tab:optimal_values_08}. 

\begin{table}[h]
	\begin{tabular}{c|cc|cc}
		& \multicolumn{2}{c|}{Bi-Laplacian prior}                            & \multicolumn{2}{c}{Ornstein-Uhlenbeck prior}                       \\ \cline{2-5} 
		\multicolumn{1}{l|}{}                                                           & \multicolumn{1}{c|}{Relative error} & $\varphi_N(\iii)$ & \multicolumn{1}{c|}{Relative error} & $\varphi_N(\iii)$ \\ \hline
		\begin{tabular}[c]{@{}c@{}}Initial guess\\ $I_0 = 0.5 \cdot 10^2,  \bo{\coeffi}_0$\end{tabular}            
		& \multicolumn{1}{c|}{$\revision{}{48.51 \%}$}              & $\revision{}{0.345}$              
		& \multicolumn{1}{c|}{$\revision{}{40.09 \%}$}              & $\revision{}{7.464}$              \\ \hline
		\begin{tabular}[c]{@{}c@{}}Optimal solution\\ $I_{\text{opt}} = 2\cdot 10^2,  \bo{\coeffi}_{\text{opt}}$\end{tabular} 
		& \multicolumn{1}{c|}{$\revision{}{\bld{39.68 \%}}$}             & $\revision{}{\bld{0.248}}$              
		& \multicolumn{1}{c|}{$\revision{}{\bld{33.17 \%}}$}              & $\revision{}{\bld{4.962}}$             
	\end{tabular}
	\caption{Reconstruction errors with initial and optimal intensity functions for $\alpha = 0.8$}
	\label{tab:optimal_values_08}
\end{table} }

To demonstrate the effect of the initial and optimal intensity functions on wave propagation, we plot snapshots of the wave field resulting from the reconstructed absorption density function $a$ at different time steps. Here, we again select the time instances $t \in \{ 0.02, 0.04, 0.06, 0.08, 0.1\}$. Results are shown in Figure~\ref{fig:state_i0_strongdamping} and Figure~\ref{fig:state_iopt_strongdamping}.

\begin{figure}[h]
	\centering
	\includegraphics[width=\textwidth]{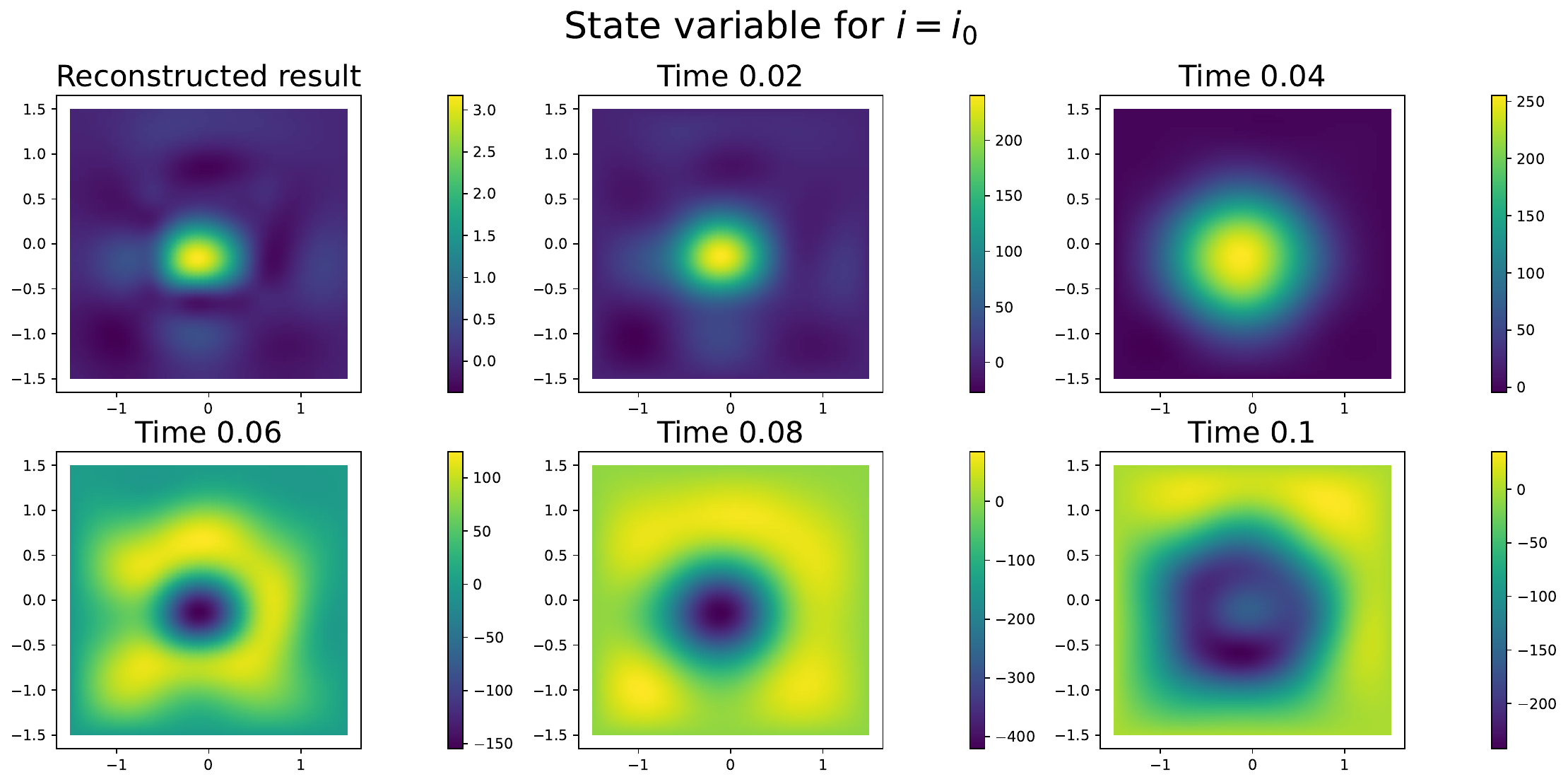}
	\caption{Snapshots of the state variable at $t \in \{ 0.02, 0.04, 0.06, 0.08, 0.1\}$ for $a$ reconstructed with $i = i_0$ in the strong damping case}
	\label{fig:state_i0_strongdamping}
\end{figure}

\begin{figure}[h]
	\centering
	\includegraphics[width=\textwidth]{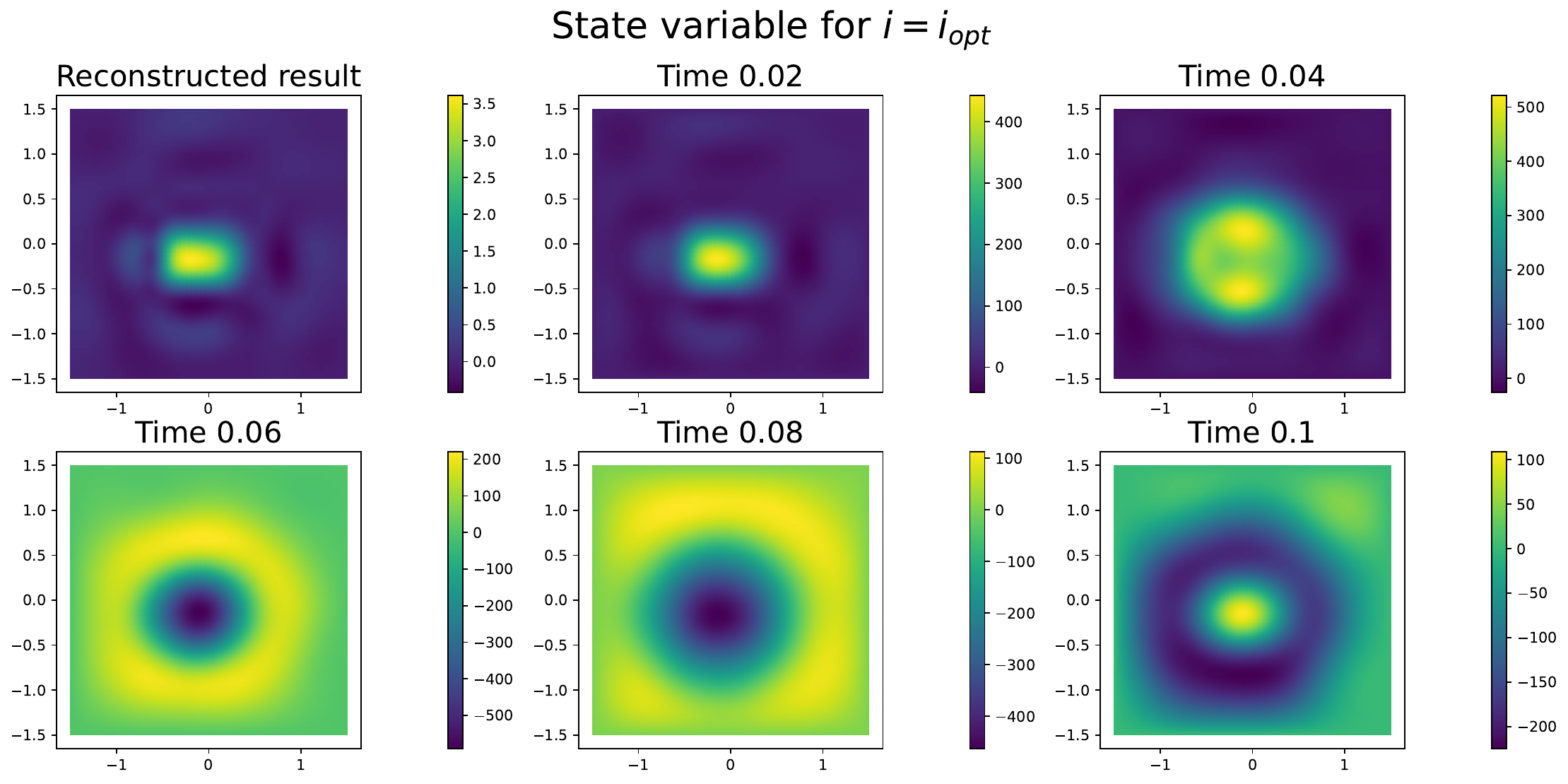}
	\caption{Snapshots of the state variable at $t \in \{ 0.02, 0.04, 0.06, 0.08, 0.1\}$ for $a$ reconstructed with $i = i_{\text{opt}}$ in the strong damping case}
	\label{fig:state_iopt_strongdamping}
\end{figure}

\subsection{Interpretation of physical units after rescaling}
 \label{sec:rescaling}
\revision{}{
Finally, we demonstrate how physical units can be interpreted in the context of the problem after appropriate rescaling. The constant $c = 15$ represents the speed of sound, corresponding to $c = 1500~\mathrm{m/s}$. Hence, one velocity unit in the simulation equals $100~\mathrm{m/s}$. The simulation time interval \((0, 0.1)\) corresponds to $10~\mu\mathrm{s} = 10^{-5}~\mathrm{s}$, implying that one time unit equals $10^{-4}~\mathrm{s}$. Hence, one space unit corresponds to $100*10^{-4}~\mathrm{m}$, that is, one centimeter and the computational domain \([ -1.5, 1.5 ] \times [ -1.5, 1.5 ]\) in simulation units corresponds to a physical domain of \([-15~\mathrm{mm}, 15~\mathrm{mm}] \times [-15~\mathrm{mm}, 15~\mathrm{mm}]\).
Our frequency unit is the reciprocal of the time unit, i.e., $10~\mathrm{kHz}$, thus our angular frequency range \(\omega \in \{20\pi, 30\pi, 40\pi, 50\pi, 60\pi\}\) corresponds to $f\in \{0.2, 0.3, 0.4, 0.5, 0.6\}\,\mathrm{MHz}$.  
}
\\
\revision{}{According to \cite{ClasonKlibanov2007} the source term takes the form 
$\alpha\frac{\beta}{c_p} \frac{d I^{\text{real}}}{dt^\text{real}}$ where $\alpha$ is the actual absorption density, $\beta$ the thermal expansion coefficient, $c_p$ the specific heat capacity. Typical values in blood  (cf. \cite[Table 1]{Pulkkinen_2014} for $\alpha$, while the other coefficients are similar to those for water at room temperature) are   
\[\alpha\sim 6\cdot 10^2 \un{m^{-1}}, \quad \beta\sim 2*10^{-4} \un{K^{-1}}, \quad c_p\sim 4*10^3 \un{J\, kg^{-1}\, K^{-1}},\] so with $J=Ws$ we have 
\begin{equation} \label{eq:3_compute_0}
\alpha\frac{\beta}{c_p}\sim 3\cdot 10^{-5}\, \un{kg\, m^{-1}\, s^{-1} \, W^{-1}}.
\end{equation}
Indeed, with $I^\text{real}$ having units $\un{W\,m^{-2}}$, thus $\frac{d I^{\text{real}}}{dt^{\text{real}}}$ having units $\un{s^{-1}\,W\,m^{-2}}$, we get the right units for this source term, namely $\un{kg\, m^{-3}\, s^{-2}}$ (matching, e.g., the unit of $-\Delta p$, which is $\un{m^{-2}\, Pa=kg\, m^{-3}\, s^{-2}}$.
Choosing the weight unit to be $10^{-12} \mathrm{\, kg}$ we have
\begin{equation} \label{eq:3_compute_1}
\begin{aligned}
        a \frac{d I}{d t} 
    &= 6 \cdot 10^2 \text{(weight unit)} \text{(length unit)}^{-3} \text{(time unit)}^{-2} = \\
    &= 6\cdot 10^4 \mathrm{kg \cdot m^{-3} s^{-2}}.
\end{aligned}
    \end{equation}
On the other hand, with $I^{\text{real}}(t)=I^{\text{real}}_0(1+\sin(\omega^{\text{real}}t))$ we have
\begin{align} \label{eq:3_compute_2}
    \max_t|\frac{d I^{\text{real}}}{dt^{\text{real}}}| = I^{\text{real}}_0\,\omega^{\text{real}}
\end{align}
Matching the product of \eqref{eq:3_compute_0} and \eqref{eq:3_compute_2} with \eqref{eq:3_compute_1} and using 
$\omega^{\text{real}}\geq 2\pi\cdot 0.2\cdot 10^6 s^{-1}$
\[
6\cdot 10^4 \,\mathrm{kg \cdot m^{-3} s^{-2}}
\geq 12\pi \, \mathrm{kg\, m^{-1}\, s^{-1} \, W^{-1}} I^{\text{real}}_0 
\]
we have
\begin{align}
    I^{\text{real}}_0 \leq \frac{5}{\pi} \cdot 10^3 \mathrm{\,W m^{-2}} \approx 160 \mathrm{\,mW / cm^2 },
\end{align}
which is the typical range of intensity values, cf., e.g. \cite{Maslov_Wang_2008}.
}

\revision{}{
With this, the frequencies in our simulation setting lie on the lower end of the typical values in PAT. In fact, higher frequencies are known to lead to computational challenges requiring dedicated numerical methods that would be beyond the scope of this paper.
}